\documentclass{article}
\usepackage[english]{babel}
\usepackage{amsfonts}
\usepackage{amsmath,amsthm,amssymb}
\usepackage{color}
\usepackage{authblk}
\usepackage{graphicx}
\usepackage{tikz}
\usepackage{stmaryrd}
\usepackage{mathrsfs,upgreek}
\usepackage{eucal}
\usepackage{enumerate}
\usepackage{changes}

\newtheorem{theorem}{Theorem}[section]
\newtheorem{lemma}[theorem]{Lemma}

\newtheorem{corollary}[theorem]{Corollary}
\theoremstyle{definition}

\newcommand{\op}[1]{\textrm{\upshape #1}}

\newcommand{\join}{\vee}

\newcommand{\meet}{\wedge}

\newcommand{\la}{\langle}
\newcommand{\ra}{\rangle}
\newcommand{\alg}[1]{{\textbf{\upshape #1}}}  %
\newcommand{\vv}[1]{\mathsf {#1}}

\renewcommand{\a}{\alpha}

\renewcommand{\d}{\delta}

\renewcommand{\th}{\theta}
\renewcommand{\o}{\omega}

\newcommand{\sse}{\subseteq}

\newcommand{\app}{\approx}
\newcommand{\HH}{{\mathbf H}}  % homomorhic images
\newcommand{\SU}{{\mathbf S}} % subalgebras
\newcommand{\PP}{{\mathbf P}}   % direct products
\newcommand{\VV}{{\mathbf V}}   % variety gen. by ...

\newcommand{\ib}{\item[$\bullet$]}

\newcommand{\Con}[1]{\operatorname{Con}(\alg #1)}

\newcommand{\imp}{\rightarrow}
\newcommand{\nneg}{\mathop{\sim}}

\makeatletter
\def\square{\RIfM@\bgroup\else$\bgroup\aftergroup$\fi
  \vcenter{\hrule\hbox{\vrule\@height.6em\kern.6em\vrule}\hrule}\egroup}
\makeatother
\begin{document}
\title{Varieties of Bounded K-lattices\footnote{This work was supported by the Italian {\em National Group for Algebra and Geometric Structures} (GNSAGA--INDAM).}}
\author{Paolo Aglian\`{o}\\
DIISM\\
Universit\`a di Siena\\
Italy\\
agliano@live.com
\and
Miguel Andr\'es Marcos\\
Facultad de Ingenier\'ia Qu\'imica\\ CONICET - Universidad Nacional del Litoral \\
Argentina\\
mmarcos@santafe-conicet.gov.ar
}
\date{}
\maketitle

\begin{abstract}
In this paper we continue to study varieties of K-lattices, focusing on their bounded versions. These (bounded) commutative residuated lattices arise from a specific kind of construction: the {\em twist-product} of a lattice. Twist-products were first considered by Kalman in 1958 to deal with order involutions on plain lattices, but the extension of this concept to residuated lattices has attracted some attention lately.
\end{abstract}

\section*{Introduction}

This paper is a natural continuation of the investigation in \cite{AglianoMarcos2020a} on varieties of K-lattices. These are commutative residuated lattices coming from a specific kind of construction: the {\em twist-product}.

The idea of considering the twist-product construction  goes back to Kalman's paper \cite{Kalman1958}, where only pure lattices were considered. The extension of this concept to residuated lattices is due to Tsinakis and Wille \cite{TsinakisWille2006}; they considered the twist-product of a residuated lattice $\alg L$ having a greatest element $\top$  such that the element $(\top,1)$ ($1$ the monoid identity) is the dualizing element relative to the natural involution. In other words for all $(x,y) \in \alg L \times \alg L^\partial$, $\nneg(x,y) = (x,y) \imp (\top,1)$ and so $((x,y) \imp (\top,1)) \imp (\top,1) = (x,y)$.

K-lattices were introduced by \cite{BusanicheCignoli2014} and they are residuated lattices that are subalgebras of the algebras obtained by applying the Tsinakis-Wille construction to {\em integral} commutative residuated lattices; in this case $\top=1$ and the dualizing pair is $(1,1)$.

The lattice of subvarieties of K-lattices have been investigated in \cite{AglianoMarcos2020a} and this paper is mostly based on the results therein, but here we choose a different starting point, as we consider twist-products coming from {\em bounded} commutative residuated lattices.

We will not try to develop a theory of bounded K-lattices, paralleling the one in \cite{AglianoMarcos2020a} and the reason is that many parts would be just a repetition of our previous work. In this paper we will focus on the differences in the various parts of the theory, quoting freely \cite{AglianoMarcos2020a} for the parts that are clearly identical. Moreover we will use the same notation for residuated lattices that are lower bounded and their {\em bounded} version, letting the context clear any possible ambiguity.

The paper is organized as follows. In Section \ref{S1.prelim}, we enumerate all general results needed to tackle the problem at hand, as well as those results from K-lattices that will be needed. In Sections \ref{S2.atomsandcovers} we describe the lower part of the lattice of subvarieties of bounded K-lattices, specifically we show that there is only one atom, describe (up to a certain extent) all finitely generated covers of the atom, and mention some infinite covers. Finally in Section \ref{S3.examples} we consider some special subvarieties of bounded K-lattices, and study the lattice of subvarieties for those cases.

\section{Preliminaries}\label{S1.prelim}

As in the case of \cite{AglianoMarcos2020a}, we first mention the classical result by B. J\'onsson.

\begin{lemma}\label{jonsson} (J\'onsson's Lemma) Let $\vv K$ be a class of algebras such that $\VV(\vv K)$ is congruence distributive; then
\begin{enumerate}
\item if $\alg A$ is a finitely subdirectly irreducible algebra in $\VV(\vv K)$, then $\alg A \in \HH\SU\PP_u(\vv K)$;
\item if $\alg A,\alg B$ are finite subdirectly irreducible algebras in $\VV(\vv K)$ then $\VV(\alg A) = \VV(\alg B)$ if and only if $\alg A\cong\alg B$.
\end{enumerate}
\end{lemma}
In particular if $\vv K$ is a finite class of finite algebras and $\VV(\vv K)$ is congruence distributive, then all the finitely subdirectly irreducible algebras in $\VV(\vv K)$ are in $\HH\SU(\vv K)$.

\medskip

A {\bf commutative residuated lattice} is a structure $\la A,\join,\meet,\cdot,\imp,1\ra$ such that
\begin{enumerate}
\item $\la A,\join,\meet\ra$ is a lattice;
\item $\la A,\cdot,1\ra$ is a commutative monoid;
\item $(\cdot,\imp)$ form a residuated pair w.r.t. the ordering, i.e. for all $a,b,c \in A$
$$
ab \le c\qquad\text{if and only if}\qquad a \le b \imp c.
$$
\end{enumerate}

We denote this variety by $\mathsf{CRL}$; if $1$ is the largest element in the ordering the lattice is said to be {\bf integral}. Commutative and integral residuated lattices form a variety which we call $\mathsf{CIRL}$. Note that algebras in $\mathsf{CRL}$ are congruence distributive, since they have a lattice reduct.

If we enlarge the type of $\mathsf{CIRL}$ with a constant $0$ and we add the axiom $0 \le x$ then we get the variety of {\bf bounded} commutative residuated lattices, denote by $\mathsf{BCRL}$ \footnote{We follow the common usage that bounded implies integral, see \cite{GJKO}.}. Any  finite algebra in $\mathsf{CIRL}$ is naturally bounded and most subvarieties of $\mathsf{CIRL}$ have their bounded version. In $\mathsf{BCRL}$ we have a natural negation given by $\neg x = x\imp 0$.

\medskip

There are two equations that result in interesting subvarieties of $\mathsf{CIRL}$ and $\mathsf{BCRL}$
\begin{align}
&(x \imp y) \join (y \imp x) \app 1.\tag{P}\\
&x(x \imp y) \app y(y \imp x); \tag{D}
\end{align}

The subvariety of $\mathsf{BCRL}$ given by equation (P) is the variety of $\mathsf{MTL}$-algebras, and the one given by equations (P) and (D) is the variety of $\mathsf{BL}$-algebras.

Algebras in $\mathsf{CIRL}$ satisfying (D) are called hoops. For hoops usually there is a convention when dealing with their bounded counterparts: we use the name {\em algebra} to identify the bounded version; so we talk about Wajsberg algebras (or $\mathsf{MV}$-algebras), product algebras, G\"odel algebras and of course the bounded version of generalized Boolean algebras are just Boolean algebras. While the term {\em basic algebra} is used sometimes in the literature the most common name for the bounded version of basic hoops is {\bf $\mathsf{BL}$-algebras} and there is no standard name for the bounded version of hoops; in this paper we call them {\bf $\mathsf{HL}$-algebras}. The theory of the bounded version of any variety is almost equal to the theory of the unbounded version.  There is one main difference though: since all the algebras are bounded they cannot be cancellative, so there is no variety corresponding to cancellative hoops. Of course cancellation comes back through the window as soon as we start talking about ordinal sums (see the definition below); in fact it is evident that any ordinal sum in which the first component is  an algebra in $\mathsf{BCRL}$ and all the  others are algebras in $\mathsf{CIRL}$ is an algebra in $\mathsf{BCRL}$.  Another general fact is the following: adding a constant changes the subalgebras of an algebra, in the sense that one has less subalgebras to consider. This usually simplifies the lattice of subvarieties: it is well known (and an easy exercise) that the lattice of subvarieties of product algebras is the three element chain (and the atom is the variety of Boolean algebras). The reader can compare this with the lattice of subvarieties of product hoops described in \cite{AFM}.

The connection however is strong; it can be proved that if $\vv V$ is a subvariety of $\mathsf{BCRL}$ then the class $\SU^0(\vv V)$ of its zero-free subreducts (that is, subalgebras as $\mathsf{CIRL}$ of the zero-free reducts of algebras in $\vv V$) is a subvariety of $\mathsf{CIRL}$ (see \cite{AFM}, where the result is stated and proved for hoops, but the divisibility is never used in the proof). But there are also more subtle differences; it is tempting to conjecture that if a sets of zero-free equations axiomatize a subvariety $\vv V$ of $\mathsf{BCRL}$ then the variety $\SU^0(\vv V)$ of its zero-free subreducts is axiomatized by the same set of equations. Though this happens very often it does not always happen, as shown in \cite{AglianoPanti1999} p. 372.

\medskip

A powerful tool for constructing is the ordinal sum. If $\alg A\in \mathsf{BCRL}$ and $\alg B \in \mathsf{CIRL}$ we  put a structure on the set  $A\setminus \{1\} \cup B\setminus\{1\} \cup \{1\}$. The ordering is given by
$$
a \le b \quad\text{if and only if} \quad \left\{
             \begin{array}{l}
               \hbox{$b=1$, or} \\
               \hbox{$a \in A_0\setminus\{1\}$ and $b \in A_1\setminus\{1\}$ or} \\
               \hbox{$a,b \in A_i\setminus\{1\}$ and $a \le_{A_i} b$, $i=0,1$.}
             \end{array}
           \right.
$$
and we define
\begin{align*}
&a \imp b = \left\{
              \begin{array}{ll}
                b, &\hbox{if $a=1$;} \\
                1, &\hbox{if $b=1$;} \\
                a \imp_{A_i} b, &\hbox{if $a,b \in A_i\setminus\{1\}$ and $a \le_{A_i} b$, $i=0,1$.}
              \end{array}
            \right.\\
&a\cdot b = \left\{
              \begin{array}{ll}
                a, & \hbox{if $a \in A_0\setminus\{1\}$ and $b \in A_1$;} \\
                b, & \hbox{if $a \in A_1$ and $b\in A_0\setminus\{1\}$}\\
                a  \cdot_{A_i} b, &\hbox{if $a,b \in A_i\setminus\{1\}$ and $a \le_{A_i} b$, $i=0,1$.}
              \end{array}
            \right.
\end{align*}
If we call $\alg A \oplus \alg B$ the resulting structure, then it is easily checked that $\alg A \oplus \alg B$ is a semilattice ordered integral and commutative residuated monoid (and so the ordinal sum of two hoops in the sense of \cite{BlokFerr2000} always exists). It might not be a residuated lattice though and  the reason is that if $1_{A}$ is not join irreducible and $\alg B$ is not bounded we run into trouble. In fact if $a,b \in A\setminus\{1\}$ and $a \join_{A} b =1_{A}$ then  the upper bounds of $\{a,b\}$ all lie in $B$; and since $B$ is not bounded there can be no least upper bound of $\{a,b\}$ in $\alg A \oplus \alg B$ and the ordering cannot be a lattice ordering. However this is the only case we have to worry about; if $1_{A}$ is join irreducible, then the problem disappears, and if $1_{A}$ is not join irreducible but $\alg B$ is bounded, say by $u$, then we can define
$$
a \join b = \left\{
              \begin{array}{ll}
                a, & \hbox{$a \in B$ and $b \in A$;} \\
                b, & \hbox{$a \in A$ and $b \in B$;} \\
                a \join_{B} b, & \hbox{if $a,b \in B$;} \\
                a \join_{A} b, & \hbox{if $a,b \in A$ and $a \join_{A} b < 1$;}\\
                u, & \hbox{if $a,b \in A$ and $a \join_{A} b = 1$;}\\
              \end{array}
            \right.
$$

We will call $\alg A \oplus \alg B$ the {\bf ordinal sum} and we will say that the ordinal sum {\bf exists} if $\alg A \oplus \alg B \in \mathsf{BCRL}$. We will now list some examples of varieties in which ordinal sums play a special role.

It is easy to see that the 3-element G\"odel chain $\alg G_3$ is isomorphic to $\alg 2\oplus\alg 2$. If we consider $\mathsf{GA}$ the variety of G\"odel algebras, it can be shown that $\mathsf{GA}$ is a locally finite variety and it is generated by all the finite G\"odel chains $\alg G_n = \alg 2 \oplus \ldots \oplus \alg 2$ ($n - 1$ summands).

More generally, algebras of the form $\alg 2 \oplus \alg D$ for some $\alg D\in \mathsf{CIRL}$ are exactly the directly idecomposable Stonean algebras (see Section \ref{Stone}).

On the other hand, subdirectly irreducible $\mathsf{BL}$-algebras are of the form $\alg A\oplus \alg D$, with $\alg A$ a (bounded) Wajsberg chain and $\alg D$ a basic hoop.

\medskip

A large class of algebras in $\mathsf{BCRL}$ can be obtained as follows. For any $\alg A \in \mathsf{CIRL}$ we construct a new algebra, called the {\bf connected rotation} of $\alg A$ in the following way; if $\alg \L_2 = \{0,\frac{1}{2},1\}$ then the universe is
$$
A^{\d_2}= (\{0\} \times A) \cup \left\{\left(\frac{1}{2},1\right)\right\} \cup (\{1\} \times A)
$$
and the operations are defined as

\begin{align*}
(x,a) \join (y,b) &= \left\{
                       \begin{array}{ll}
                         (1,a \join b) & \hbox{if $x=y=1$;} \\
                         (0,a \meet b), & \hbox{if $x=y=0$;} \\
                         (y,b), & \hbox{if $x <y$.}
                       \end{array}
                     \right.\\
(x,a) \meet (y,b) &= \left\{
                       \begin{array}{ll}
                         (1,a \meet b) & \hbox{if $x=y=1$;} \\
                         (0,a \join b), & \hbox{if $x=y=0$;} \\
                         (y,b), & \hbox{if $x <y$.}
                       \end{array}
                     \right.\\
(x,a)(y,b) &= \left\{
                \begin{array}{ll}
                  (1,ab), & \hbox{if $x=y=1$;} \\
                  (xy,1), & \hbox{if $x,y \ne 1$;} \\
                  (y,b), & \hbox{if $x=1$ and $y=\frac{1}{2}$;} \\
                  (y,a \imp b), & \hbox{if $x=1$ and $y=0$.}
                \end{array}
              \right.\\
(x,a) \imp (y,b) &= \left\{
                      \begin{array}{ll}
                        (1,a \imp b), & \hbox{if $x=y=1$;} \\
                        (1,b \imp a), & \hbox{if $x=y=0$;} \\
                        (0,ab), & \hbox{if $x=1$ and $y=0$;} \\
                        (1,1), & \hbox{if $x <y$;} \\
                        (x \imp y,1), & \hbox{otherwise.}
                      \end{array}
                    \right.
\end{align*}

It is not hard to check that $\alg A^{\d_2} \in \mathsf{BCRL}$  and it is also involutive; moreover the set $A^{\d_2} \setminus \left\{\left(\frac{1}{2},1\right)\right\}$ is a subalgebra of $\alg A^{\d_2}$ that we will call the {\bf disconnected rotation} of $\alg A$ and denote by $\alg A^{\d_1}$. Connected and disconnected rotations of algebras in $\mathsf{CIRL}$ have been studied in \cite{BMU2018}, \cite{AFU2017} and \cite{AglianoUgolini2019a}.

On one hand, $\alg 2^{\d_1}$, the disconnected rotation of the 2-element Boolean algebra, is the 4-element nilpotent minimum chain $\alg N_4$. In general, the variety of nilpotent minimum algebras is generated by connected and disconnected rotations of finite G\"odel chains (see Section \ref{nilpotent}).

On the other hand, the Chang algebra $\alg \L_1^\o$ (see Section \ref{mv}) is isomorphic to the disconnected rotation of the cancellative hoop $\alg C_\omega$.

\medskip

Let $\alg A \in \mathsf{CIRL}$; the {\bf K-expansion} $K(\alg A)$ of $\alg A$ is a structure whose universe is $A\times A$ and the operations are defined as:
\begin{align*}
&\la a,b\ra \join \la c,d\ra := \la a \join c, b \meet d\ra \\
&\la a,b\ra \meet \la c,d\ra := \la a \meet  c, b \join d\ra\\
&\la a,b\ra \la c,d\ra := \la ac,( a \imp d) \meet (c \imp b)\ra\\
&\la a,b\ra \imp \la c,d\ra := \la (a \imp c) \meet (d \imp b),ad\ra.
\end{align*}

\begin{theorem}\cite{BusanicheCignoli2014} For every $\alg A \in \mathsf{CIRL}$, $K(\alg A)$ is a commutative residuated lattice that is also
\begin{enumerate}
\item  $1$-involutive;
\item {\bf $1$-distributive}, i.e. it satisfies both distributive laws for lattices whenever at least one of the elements is equal to $1$.
\item If we set $\nneg x = x \imp 1$ then it satisfies the equations
\begin{align*}
& xy \meet 1 \app (x\meet 1)(y\meet 1) \tag{K1}\\
&((x \meet 1) \imp y) \meet ((\nneg y \meet 1) \imp \nneg x) \app x \imp y.  \tag{K2}
\end{align*}
\end{enumerate}
\end{theorem}

A {\bf Kalman lattice} or just {\bf K-lattice} is a a commutative integral residuated lattice that is $1$-involutive, $1$-distributive and satisfies (K1) and (K2); the variety of K-lattices is denoted by $\mathsf{KL}$.

\medskip

If we start with an algebra $\alg A\in\mathsf{BCRL}$, then $K(\alg A)$ will have $(0,1)$ as a lower bound, and $(1,0)=\sim(0,1)$ as an upper bound. Thus we define \textbf{Bounded K-lattices} as the variety of K-lattices with an extra constant $0$ and the axiom $0 \le x$ (observe that they will also be upper bounded by $\top = \sim 0$). We will denote the variety of bounded K-lattices by $\mathsf{BKL}$.

Most of the basic definitions about K-lattices in \cite{AglianoMarcos2020a} can be formulated in the bounded case as well. In particular if $\alg B \in \mathsf{BKL}$ we denote by $\alg B^-$ the algebra whose universe is
$\{b \in B: b \le 1\}$ endowed with same lattice operations, constants $0,1$ and multiplication as $\alg B$ and a new implication $a \imp^- b = a \imp b \meet 1$. It is clear that $\alg B^- \in \mathsf{BCRL}$ and moreover $\Con B$ and
$\op{Con}(\alg B^-)$ are isomorphic via the mappings
$$
\a \longmapsto \a^-= \a \cap (B^- \times B^-) \qquad  \th \longmapsto \op{Cg}_\alg B(\th).
$$
In particular for each $\a \in \Con B$ there is an $\a^- \in \op{Con}(\alg B^-)$ such that $(\alg B/\a)^- \cong \alg B^-/\a^-$ (see \cite{HartRafterTsinakis2002} for details. Moreover we have two lemmas whose proofs are identical
to the analogous one for the unbounded case (the first appears in \cite{AglianoMarcos2020a} and the second    in \cite{BusanicheCignoli2014}).

\begin{lemma}\label{subdir}  Let $\alg A,\alg B  \in \mathsf{BCRL}$ and let $(\alg A _i)_{i\in I} \sse \mathsf{BCRL}$. Then
\begin{enumerate}
\item if $\alg A \le \alg B$, then $\alg A^- \le \alg B^-$;
\item $(\Pi_{i \in I} \alg A_i)^- \cong \Pi_{i \in I} \alg A^-$.
\end{enumerate}
\end{lemma}

\begin{lemma}\cite{BusanicheCignoli2014}\label{inclusions} If $\alg L \in \mathsf{BCRL}$ then $\alg L \cong K(\alg L)^-$. If $\alg A \in \mathsf{BKL}$, then $f:a \longmapsto (a \meet 1, \nneg a \meet 1)$ is an embedding of $\alg A $ in $K(\alg A^-)$.
\end{lemma}

By the same token, all the properties of the operator $K$ transfer to the bounded case without any change. As in \cite{AglianoMarcos2020a}, a subvariety $\vv W$ of $\vv BKL$ is a \textbf{Kalman variety} if $\vv W = K(\vv W^-)$.

\begin{lemma}\label{techlemma}Let $\vv V$ be any subvariety of $\mathsf{BCRL}$,  $\vv W$ any subvariety of $\mathsf{BKL}$ and $\vv K$ any subclass of $\mathsf{BCRL}$:
\begin{enumerate}
 \item $K(\vv V)$ is a subvariety of $\mathsf{BKL}$ and $K(\vv V) = \{\alg A\in \mathsf{KL}: \alg A^- \in \vv V\}$;
 \item  $K(\VV(\vv K)) = \VV(K(\vv K))$;
\item $\HH\SU\PP_u(K(\vv K)) \sse \SU K(\HH\SU\PP_u(\vv K))$;
 \item  $K(\vv V)^-= \vv V$ and $\vv W \sse K(\vv W^-)$;
\item  $K(\vv W^-)$ is the smallest Kalman variety containing $\vv W$;
\item $K:\Lambda(\mathsf{BCRL})  \longmapsto\Lambda(\mathsf{BKL})$ is a lattice homomorphism;
\item  $\vv W^- \sse \vv V$ if and only if $\vv W \sse K(\vv V)$,
 hence\footnote{i.e. the operators $K,^-$ form a {\em left adjoint pair} between $\Lambda(\mathsf{BCRL})$ and $\Lambda(\mathsf{BKL})$.}
$$
\vv W^- = \bigwedge\{\vv U: \vv W \sse K(\vv U)\}.
$$
\item $K$ is also injective, i.e. it is an embedding.
\end{enumerate}
\end{lemma}

\medskip

If ${\alg A}\in\mathsf{BKL}$ and $\alg B$ is a subalgebra of $\alg A$, we say that $\alg B$ is an \textbf{admissible subalgebra} if  $B^-=A^-$. Since it is evident that the intersection of any family of admissible subalgebras is still admissible, admissible subalgebras form a complete meet-semilattice (indeed algebraic) under inclusion.

If $\alg A\in \mathsf{BCRL}$, we define $K_0(\alg A)$ as the minimal admissible subalgebra of $K(\alg A)$. This notation will come in handy when we look at some of the examples.

In some special cases admissible subalgebras can be totally characterized.

\medskip

An algebra $\alg A \in \mathsf{BCRL}$ is {\bf involutive} if for all $a \in A$, $\neg\neg a =a$; in this case we may define $a \oplus b = (a \imp 0)\imp b$.

\begin{theorem}\label{kalmanwajsberg} Let $\alg A \in \mathsf{BCRL}$ be  involutive; then there is a one to one and onto correspondence between the lattice filters of $\alg A$ and the admissible subalgebras of $K(\alg A)$. More precisely, if $F$ is a lattice filter of $\alg A$, then
$$
K(\alg B,F) =\{(a,b) \in K(\alg B): a \oplus b \in F\} \le K(\alg B)
$$
is admissible. Conversely if $\alg A\in \mathsf{BKL}$ is such that $\alg A^-$ is involutive, then
$$
F= \{((a \meet 1) \imp_1 0) \imp_1 \nneg a: a \in A\}
$$
is a lattice filter of $\alg A^-$ and $K(\alg A^-,F)$ is an admissible subalgebra of $K(\alg A^-)$.
\end{theorem}

Wajsberg algebras and rotations (both connected and disconnected) are examples of involutive $\mathsf{BCRL}$.

\medskip

 Let now $\alg H$ be a  Heyting algebra; an element $a\in \alg H$ is {\bf dense} if $\neg a=0$. A filter $F$ of $\alg H$ is {\bf regular} of it contains all the dense elements; it is easy to see that $F$ is regular if and only if $\alg H/F$ is a Boolean algebra.

\begin{theorem}\label{kalmanheyting} Let $\alg H$ be a Heyting algebra; then there is a one to one and  correspondence between the regular filters of $\alg H$ and the algebras in $\mathsf{KL}$ whose negative cone is isomorphic with $\alg H$. More precisely, if $F$ is a regular filter of $\alg H$, then
$$
K(\alg H,F)= \{(a,b) \in K(\alg H):  a \join b \in F\} \le K(\alg H)
$$
and $K(\alg H,F)^- \cong \alg H$. Conversely if $\alg A\in \mathsf{BKL}$ is such that  $\alg A^-$ is a Heyting algebra, then
$F=\{(a \join \nneg a) \meet 1: a \in A\}$ is a regular filter of $\alg A^-$ such that $K(\alg A^-,F) \cong \alg A$.
\end{theorem}

For a proof of the two previous results we quote \cite{BusanicheCignoli2014}.

\medskip

We also have the following result, concerning ordinal sums.

\begin{theorem}\label{admissibleordinalsum}(\cite{AglianoMarcos2020a}, Theorem 3.4) The admissible subalgebras of $K({\alg A}\oplus {\alg B})$ are in one to one correspondence with the admissible subalgebras of $K({\alg A})$. Moreover, if $\alg S$ is an admissible subalgebra of $K({\alg A})$, then $T_S^\alg B=S\cup (A\times B)\cup (B\times A)\cup (B\times B)$ is the universe of an admissible subalgebra of $K({\alg A}\oplus {\alg B})$.
And if $\alg T$ is an admissible subalgebra of $K({\alg A}\oplus {\alg B})$, then $S_T=T\cap A\times A$ is the universe of an admissible subalgebra of $K({\alg A})$ that satisfies $T_{S_T}^\alg B=T$.
\end{theorem}

As a particular case, we can completely describe the admissible subalgebras of $K({\alg A}\oplus {\alg B})$ if $\alg A\in\mathsf{BCRL}$ is either involutive or Heyting, using Theorems \ref{kalmanwajsberg} and \ref{kalmanheyting}.

\section{Atoms and covers}\label{S2.atomsandcovers}

If $\vv V$ is any variety we denote by $\Lambda(\vv V)$ its lattice of subvarieties; a variety $\vv W \in \Lambda(\vv V)$ that is a cover of an atom is called {\bf almost minimal}. The lattice of subvarieties of (non necessarily bounded) K-lattices  has been investigated at length in \cite{AglianoMarcos2020a}.  If $\mathbf 2$ is the two element Boolean algebra and  then $K(\mathbf 2)$ has four elements  and it is the only four element bounded K-lattice, so it makes sense to denote it by $\alg K_4$.  Moreover  if $2=\{0,1\}$, then  $\{(0,1),(1,1),(1,0)\}$ is the universe of admissible subalgebra $\alg K_3$ of $K(\mathbf 2)$. As in the unbounded case \cite{BusanicheCignoli2014}  $\alg K_3$ is the only totally ordered algebra in $\mathsf{BKL}$ and therefore
$\VV(\alg K_3)$ is the only representable subvariety of $\mathsf{BKL}$.  $K(\VV(\mathbf 2))$ is of course the variety generated by the {\em kalmanization} of the variety $\mathsf{BA}$ of Boolean algebras; and in the same way that \cite{AglianoMarcos2020a} we can check that $\VV(\alg K_3)$ is not a Kalman variety.
Here is the first result that is different for bounded K-lattices.

\begin{theorem}\label{onlyatom} $\VV(\alg K_3)$ is the only  atom in $\Lambda(\mathsf{BKL})$.
\end{theorem}

There is another subtler difference of behavior in bounded K-lattices. Here, the element $(0,0)$ always exists in $K(\alg A)$ for $\alg A\in \mathsf{BCRL}$, and can be described equationally.

\begin{lemma} If $\alg A\in \mathsf{BKL}$, then there exists at most one element ${\bf o}\in \alg A$ such that
\begin{itemize}
	\item $\sim {\bf o} = {\bf o}$,
	\item ${\bf o}\wedge 1=0$.
\end{itemize}
\end{lemma}
\begin{proof}If $z\in A$ satisfies both $\sim z=z$ and $z\wedge 1=0$, then clearly $z\wedge 1={\bf 0}\wedge 1$ and $\sim z\wedge 1=\sim{\bf o}\wedge 1$, and in Kalman lattices these two equations imply that $z={\bf o}$ (see \cite{BusanicheCignoli2014}).
\end{proof}

\begin{lemma}\label{abovek3} If $\alg A\in \mathsf{BKL}$, the following are equivalent
\begin{enumerate}
	\item $\alg A \cong K(\alg A^-)$,
	\item $\alg K_4\leq \alg A$,
	\item there exists ${\bf o}\in A$ such that $\sim {\bf o} = {\bf o}$ and ${\bf o}\wedge 1=0$.
\end{enumerate}
\end{lemma}
\begin{proof}$1.\Rightarrow 2$ and $2.\Rightarrow 3.$ are immediate, as $(0,0)$ satisfies the equations for $\bf o$. To show that $3.\Rightarrow 1.$, observe that we always have $\alg A \leq K(\alg A^-)$ as $x\mapsto (x\wedge 1, \sim x\wedge 1)$ is an embedding, by Lemma \ref{inclusions}. But if $3.$ holds, it is also onto: if $a,b\in A^-$, then $z=a\vee (\sim b\wedge {\bf o})$ satisfies $z\wedge 1 = a$ and $\sim z\wedge 1=b$.
\end{proof}

Lemma \ref{abovek3} implies that Kalman subvarieties of $\mathsf{BKL}$ will always be above $\VV(\alg K_4)$. More precisely:

\begin{theorem} $\VV(\alg K_4)=K(\mathsf{BA})$ is almost minimal in $\Lambda(\mathsf{BKL})$; moreover it is the only almost minimal Kalman variety.
\end{theorem}

\subsection{Finitely generated almost minimal varieties}

In Section 5 of \cite{AglianoMarcos2020a}, we introduced a subclass of $\mathsf{CIRL}$ from which we could construct finitely generated almost minimal varieties in $\Lambda(\mathsf{KL})$. An algebra $\alg A \in \mathsf{CIRL}$ is {\bf tight} if
\begin{enumerate}
\ib $|A|> 2$;
\ib $\alg A$ is bounded by $0$ and any element different from $0,1$ generates $\alg A$.
\end{enumerate}

It turns out that finite tight algebras describe almost all finitely generated almost minimal varieties in $\mathsf{CIRL}$ :

\begin{theorem}\cite{AglianoGalatosMarcos2020}\label{alltights} If $\alg A\in\mathsf{CIRL}$ is a finite subdirectly irreducible algebra generating an almost minimal variety, then either $\alg A$ is tight, or it isomorphic with the 0-free reduct of either $\alg G_3$ or $\alg N_4$.\end{theorem}

The definition of tight algebras would make sense also for algebras in $\mathsf{BCRL}$ but there is a relevant difference: in $\mathsf{CIRL}$ every filter is a subalgebra, so the second condition above implies that any tight algebra is simple. Since in $\mathsf{BCRL}$ a filter is not in general a subalgebra, simplicity may not hold. Therefore we give a different and more general definition.

We say an algebra $\alg A \in \mathsf{BCRL}$ is \textbf{rigid} if:
\begin{itemize}
	\item $|A|>2$;
	\item $\alg A$ is subdirectly irreducible;
	\item $\alg A$ has no proper subalgebras different from $\{0,1\}$;
	\item for all proper nontrivial $\theta\in \Con A$, $\alg A/\theta \cong \alg 2$.
\end{itemize}

It is clear from the definition that the algebras $\alg G_3$ and $\alg N_4$ are rigid.  Observe that any stiff algebra in the sense of  \cite{GJKO,KKU2006} is rigid; the following results will show the importance of finite rigid algebras.

\begin{lemma}\label{stiff} Let $\alg A \in \mathsf{BCRL}$ be rigid; then $\alg A$ has at most one proper nontrivial congruence, i.e. $\Con A$ is either the two or three-element chain.\end{lemma}
\begin{proof}If $\mu$ is the monolith and it is a proper congruence (otherwise we are done), then $\alg A/\mu \cong \alg 2$. If $\theta\supsetneq \mu$, it must be the congruence $\theta=A\times A$, as it must have a pair $(a,b)$ such that $(a,0),(b,1)\in\mu$.\end{proof}

\begin{theorem}\label{stiff2} The following results hold in $\mathsf{BCRL}$.
\begin{enumerate}
	\item If $\alg A\in\mathsf{BCRL}$ is a finite subdirectly irreducible algebra, then $\VV(\alg A)$ is an almost minimal variety if and only if $\alg A$ is rigid.
	\item If $\alg A\in\mathsf{BCRL}$ satisfies that the 0-free reduct of $\alg A$ is a tight algebra in $\mathsf{CIRL}$, then it is rigid.
\end{enumerate}
\end{theorem}
\begin{proof}The first part is immediate, as the definition of rigid algebras and Lemma \ref{stiff} characterize the finitely generated covers of $\VV(\alg 2)=\mathsf{BA}$.

For the second part, it is clear that if the 0-free reduct of $\alg A$ is tight, then any $a\neq 0,1$ generates the whole $A$ as an algebra in $\mathsf{CIRL}$, but as it contains 0 it implies that $\alg A$ has no subalgebras other than $\alg A$ and $\{0,1\}$ as an algebra in $\mathsf{BCRL}$. Moreover, as filters are subalgebras in $\mathsf{CIRL}$, we have that $\alg A$ is subdirectly irreducible and does not have any non-trivial proper congruences. Therefore $\alg A$ is rigid.
\end{proof}

The converse of the second part of Theorem \ref{stiff2} is false, however, as there are new covers of the atom $\VV(\alg 2)=\mathsf{BA}$. For instance, the reader can verify that the algebra described in Figure \ref{stiffnottight} is rigid, but it is clearly not tight as an algebra in $\mathsf{CIRL}$ (and it is neither $\alg G_3$ nor $\alg N_4$, so by Theorem \ref{alltights} it is a distinct cover).

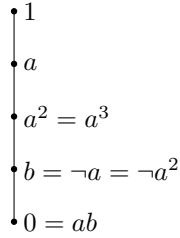
\begin{figure}[htbp]
\begin{center}
\begin{tikzpicture}[scale=0.7]
\draw (0,0) -- (0,1) -- (0,2) -- (0,3) -- (0,4);
\draw[fill] (0,0) circle [radius=0.05];
\draw[fill] (0,1) circle [radius=0.05];
\draw[fill] (0,2) circle [radius=0.05];
\draw[fill] (0,3) circle [radius=0.05];
\draw[fill] (0,4) circle [radius=0.05];
\node[right] at (0,0) {$0=ab$};
\node[right] at (0,1) {$b=\neg a =\neg a^2$};
\node[right] at (0,2) {$a^2=a^3$};
\node[right] at (0,3) {$a$};
\node[right] at (0,4) {$1$};
\end{tikzpicture}
\end{center}
\caption{An algebra that generates a cover of $\VV(\alg 2)$ in $\mathsf{BCRL}$ but not in $\mathsf{CIRL}$.\label{stiffnottight}}
\end{figure}

We will now proceed to show the importance of rigid algebras for $\mathsf{BKL}$.

\begin{lemma}\label{stiffk3} Let $\alg A \in \mathsf{BCRL}$ be rigid having a non-trivial proper congruence; then there exists a proper admissible subalgebra of $K(\alg A)$ such that its only proper nontrivial quotient is isomorphic to $\alg K_3$.
\end{lemma}
\begin{proof} By Lemma \ref{stiff}, let $\mu$ be the proper monolith of $\alg A$ and let $\th = \op{Cg}_{K(\alg A)}(\mu)$, so that $\th^-=\mu$. Clearly $K(\alg A)/\theta \cong \alg K_4$, and consider the set
$$
B = \{(a,b)\in K(\alg A): (a,b)/\theta \in \alg K_3\} = \{(a,b)\in K(\alg A): ((a,b),(0,0))\not\in\theta\}.
$$
Then $B$ is nonempty ($(0,1),(1,0),(1,1)\in B$); clearly it is the universe of an admissible subalgebra $\alg B$ of $K(\alg A)$ and moreover $\alg B/(\theta\cap B\times B) \cong \alg K_3$.  As $\theta\cap B\times B$ is the only proper nontrivial congruence of $\alg B$, the Lemma holds.
\end{proof}

\begin{theorem}\label{cover}  Let $\alg A \in \mathsf{BKL}$ be a finite subdirectly irreducible algebra that generates an almost minimal variety different from $\VV(\alg K_4) = K(\mathsf{BA})$. Then  $\alg A^-$ is rigid.

On the other hand, let $\alg A$ be a finite rigid algebra in $\mathsf{BCRL}$ such that $0$ is meet irreducible. Then $K(\alg A)$  has a subalgebra that generates an almost minimal variety  different from $K(\mathsf{BA})$.
\end{theorem}
\begin{proof} Suppose that $\alg A$ is a finite subdirectly irreducible algebra generating an almost minimal variety different from $K(\mathsf{BA})$. Clearly $|A^-|> 2$, and any proper subalgebra of $\alg A$ must be isomorphic with $\alg K_3$. This implies that $\alg A^-$ cannot have proper subalgebras different from $\{0,1\}$. Consider now a proper nontrivial $\th \in \op{Con}(\alg A^-)$ and let $\a = \op{Cg}_\alg A(\th)$. Then $\a$ is proper nontrivial in $\Con A$ and hence
 we must have $\alg A/\th \cong \alg K_3$.  This implies that $(\alg A/\a)^- = \alg A^-/\th =\mathbf 2$ so $\alg A^-$ is rigid.

Conversely, let $\alg A$ be a finite rigid algebra such that $0$ is meet irreducible. Since $\alg A$ is finite and subdirectly irreducible, so is $K(\alg A)$.
As $\alg A$ is rigid, it does not have proper subalgebras different from $\{0,1\}$, so the subalgebras of $K(\alg A)$ are only $\alg K_3$, $\alg K_4$ and the admissible subalgebras.
As $0$ is meet irreducible in $\alg A$, $K(A)\setminus\{(0,0)\}$ is the universe of an admissible subalgebra $\alg B$ of $K(\alg A)$ (see Lemma 4.8 in \cite{AglianoMarcos2020a}) that does not contain $(0,0)$, so $\alg K_4$ is not a subalgebra of $\alg B$ (see Lemma \ref{abovek3}).

By Lemma \ref{stiffk3} there is an admissible subalgebra $\alg C \le K(\alg A)$ with the property that its only nontrivial quotient is $\alg K_3$ (if it has one).

Therefore the minimal admissible subalgebra $K_0(\alg A)$ satisfies that its only subalgebras are $\alg K_3$ and $K_0(\alg A)$, and its only non-trivial quotient (if any) is $\alg K_3$. By J\'onsson's Lemma we conclude that $K_0(\alg A)$ generates a cover of the atom in $\Lambda(\mathsf{BKL})$.

\end{proof}

\subsection{Other almost minimal varieties}

We will now consider some examples of non-finitely generated covers of $\VV(\alg K_3)$.

\begin{lemma}\label{firstorder} Let $\alg A\in \mathsf{BKL}$ satisfy that $\alg K_4\not\leq \alg A$. Then each $\alg B\in \SU\PP_u(\alg A)$ satisfies $\alg K_4\not\leq \alg B$.\end{lemma}
\begin{proof}
By Lemma \ref{abovek3}, $\alg K_4\not\leq \alg A$ is equivalent to the non-existence of an element $\bf o$ such that ${\bf o} = \sim {\bf o}$ and ${\bf o}\wedge 1 =0$. But this is  first-order definable by the sentence given by $\varphi(x): (x\app x\to 1)$ and $\psi(x): (x \meet 1 \app 0)$
\begin{align*} 	\forall\,x (\neg\varphi(x) \join \neg\psi(x)). \end{align*}
Therefore if it is true for $\alg A$, then it is true for any ultrapower, and clearly for any subalgebra $\alg B\in \SU\PP_u(\alg A)$.
\end{proof}

We will use this Lemma to construct two covers.

\medskip

First we will show that $K_0(\mathbf 2 \oplus \alg C_\o)$ generates an almost minimal variety different from $K(\mathsf{BA})$. Observe that by Theorem \ref{admissibleordinalsum}, we have that the universe of $K_0(\mathbf 2 \oplus \alg C_\o)$ is $K(\mathbf 2 \oplus \alg C_\o)\setminus\{(0,0)\}$. Therefore $\alg K_4\not\leq K_0(\mathbf 2 \oplus \alg C_\o)$ and by Lemma \ref{firstorder} we have that $\alg K_4\not\in \SU\PP_u(K_0(\mathbf 2 \oplus \alg C_\o))$.

Now, in $2 \oplus \alg C_\o$ we have the first-order formula
\begin{align*} 	\forall\,x ((x\app 0) \join (\neg x \app 0)), \end{align*}
Therefore, if $\alg A\in \SU\PP_u(K_0(2 \oplus \alg C_\o))$ and $\theta$ is a congruence in $\alg A$ with $\alg A/\theta \leq \alg K_4$, $F=\theta\cap A^-$ is a filter in $\alg A^-$ such that $\alg A^-/F\cong \alg 2$. Then if $a\in A^-$ satisfies $a\sim_F 0$, $\neg a\in F$ and therefore $a = 0$ by the first-order formula preserved in $\PP_u(2 \oplus \alg C_\o)$. Thus as $(0,0)$ is not an element of $\alg A$, we have that $\alg A/\theta \cong \alg K_3$.
This shows that $\alg K_4\not\in \HH\SU\PP_u(K_0(2 \oplus \alg C_\o))$, and therefore $\alg K_4\not\in \VV(K_0(2 \oplus \alg C_\o))$.

To show that the variety is minimal, if $\alg B \in \HH\SU\PP_u(K_0(2 \oplus \alg C_\o))$, by using Lemma \ref{subdir} and the well-known properties of cancellative hoops we can deduce that either $\alg B^- \cong \mathbf 2$ or $\alg B^- \cong \mathbf 2 \oplus \alg A$ for some totally ordered cancellative hoop $\alg A$. In case $\alg B^-\cong \mathbf 2$, then
$\alg B \cong \alg K_3$ (as we showed that $\alg K_4\not\in \HH\SU\PP_u(K_0(2 \oplus \alg C_\o))$). Otherwise  $\alg B \le K(\mathbf 2 \oplus \alg A)$ and $(0,0) \notin B$; but then $\alg B \cong K_0(\mathbf 2 \oplus \alg C)$ for some subalgebra $\alg C \le \alg B$ and $\mathbf 2 \oplus \alg C_\o \le \mathbf 2 \oplus \alg C$. It follows that $K_0(\mathbf 2 \oplus \alg C_\o) \le K(\mathbf 2 \oplus \alg C) \cong \alg B$. This is enough to prove that $K_0(\mathbf 2 \oplus \alg C_\o)$ generates an almost minimal variety not above $K(\mathsf{BA})$.

\medskip

We will now use a similar argument to show that $K_0(\alg \L_1^\o)$ generates an almost minimal variety above $\VV(\alg K_3)$. Observe that the Chang algebra satisfies $\alg \L_1^\o \cong (\alg C_\o)^{\d_1}$, that is the disconnected rotation of the cancellative hoop $\alg C_\o$, and we can identify $\alg C_\o$ with the radical (and only filter) of $\alg \L_1^\o$. It can be shown that the universe of $K_0(\alg \L_1^\o)$ is $K(\alg \L_1^\o)\setminus\{(a,b):\neg a,\neg b\in C_\o\}$.
It is clear that $\alg K_4\not\leq K_0(\alg \L_1^\o)$ and by Lemma \ref{firstorder} we have that $\alg K_4\not\in \SU\PP_u(K_0(\alg \L_1^\o))$.

Now, if $\alg A\in \SU\PP_u(K_0(\alg \L_1^\o))$ and $\theta$ is a congruence in $\alg A$ with $\alg A/\theta \leq \alg K_4$, $F=\theta\cap A^-$ is a filter in $\alg A^-$ such that $\alg A^-/F\cong \alg 2$. As $\alg \L_1^\o$ satisfies the first-order formula
\begin{align*} 	\forall\,x ((x^2\app 0) \join (\neg x \imp x\app 1)), \end{align*}
then if $a\in A^-$ satisfies $a\sim_F 0$, it must be $a^2=0$. But $K_0(\alg \L_1^\o)$ satisfies the first-order formula
\begin{align*} 	\forall\,x (((x\wedge 1)^2\app 0) \Rightarrow \neg(((x\imp 1)\wedge 1)^2\app 0)), \end{align*}
there does not exist $x\in \alg A$ such that $x/\theta = (0,0)$, so $\alg A/\theta \cong \alg K_3$.
This shows that $\alg K_4\not\in \HH\SU\PP_u(K_0(\alg \L_1^\o))$, and therefore $\alg K_4\not\in \VV(K_0(\alg \L_1^\o))$.

To show that the variety is minimal, suppose that $\alg B \in \HH\SU\PP_u(K_0(\alg \L_1^\o))$; certainly $K_0(\alg L_1^\o)$  satisfies a first order sentence that says that its negative part (i.e. $\alg L_1^\o$) is totally ordered. It follows that $\alg B$ satisfies the same sentence, i.e. $\alg B^-$ is totally ordered as well.  Since $\alg B^-$ belongs to the variety generated by the Chang algebra,  either $\alg B^- \cong \mathbf 2$ or $\alg B^- \cong (\alg A)^{\d_1}$ for some totally ordered cancellative hoop $\alg A$. In case $\alg B^-\cong \mathbf 2$, then
$\alg B \cong \alg K_3$ (as we showed that $\alg K_4\not\in \HH\SU\PP_u(K_0(\alg \L_1^\o))$). Otherwise $\alg B \le K((\alg A)^{\d_1})$ and $(0,0) \notin B$; but then $K_0(\alg \L_1^\o) \le \alg B$. This is enough to prove that $K_0(\alg \L_1^\o)$ generates an almost minimal variety not above $K(\mathsf{BA})$.

\section{Lattices of subvarieties}\label{S3.examples}

In this section we will consider the lattices of subvarieties of $K(\vv V)$ for some specific $\vv V \sse \mathsf{BCRL}$. Some results are very similar to the ones obtained for the unbounded cases in \cite{AglianoMarcos2020a}, some are substantially different and some are totally new in the sense that deal with varieties that have no corresponding known unbounded counterpart. We will omit all the proofs that are simple rewritings of the proofs in \cite{AglianoMarcos2020a} and we will illustrate in more details the {\em new} cases.

\subsection{Bounded Basic K-Lattices}

In this Section we will deal with different suvbarieties of $K(\mathsf{BL})$ but first we will mention some basic results.

The finitely generated minimal varieties are $K(\mathsf{BA})$, $\VV(K_0(\alg G_3))$ and all $\VV(K_0(\alg \L_p))$ for $p$ prime. The non-finitely generated ones are $\VV(K_0(\mathbf 2 \oplus \alg C_\o))$ and $\VV(K_0(\alg \L_1^\o))$.

Observe also that if $\alg A\in\mathsf{BL}$ is subdirectly, then $\alg A\cong \alg B\oplus \alg D$, where $\alg B$ is a Wajsberg algebra (actually a chain) and $\alg D$ is a basic hoop. Then all the admissible subalgebras of $K(\alg A)$ are in one to one correspondence with the admissible subalgebras of $K(\alg B)$, which in turn are in one to one correspondence with the lattice filters of $\alg B$ by Theorems \ref{admissibleordinalsum} and \ref{kalmanwajsberg}.

\subsubsection{Bounded Wajsberg K-lattices}\label{mv}

Let $\mathsf{WA}$ be the variety of Wajsberg algebras; a bounded K-lattice $\alg A$ is a {\bf bounded Wajsberg K-lattice} if $\alg A^- \in \mathsf{WA}$. The lattice $\Lambda(\mathsf{WA})$ is well known \cite{Komori1981} and it is
simpler than the lattice of subvarieties of Wajsberg hoops \cite{AglianoPanti1999}; this simplification is reflected also in the structure of $\Lambda(K(\mathsf{WH}))$.
Since Wajsberg algebras are involutive, if $\alg \L_n$ is the $n+1$-element Wajsberg chain, then its admissible subalgebras can be described exactly as in \cite{AglianoMarcos2020a}, Section 5: if $a$ is the only coatom of $\alg \L_n$ then
$$
K_{m,n} = \{(u,v): u \oplus v \ge a^m\}
$$
is the universe of an admissible subalgebra $\alg K_{m,n}=K(\alg \L_n,F)$ of $K(\alg \L_n)$, for $F=\langle a^m\rangle$ the filter generated by $a^m$. Now similarly to \cite{AglianoMarcos2020a} we can show:

\begin{theorem} The only almost minimal varieties in $\Lambda(K(\mathsf{WA}))$ are
 $K(\mathsf{BA})$, $\VV(\alg K_{0,p})$ for $p$ prime,
and $\VV(K_0(\alg \L_1^\o))$.
\end{theorem}

We define
$$
\alg \L^\o_n = \Gamma(\mathbb Z \times^l\mathbb Z, (n,0)),
$$
where $\times^l$ is the lexicographic product and $\Gamma$ is the Mundici functor \cite{Mun1986}.

Each proper subvariety of $\mathsf{WA}$ has only finitely many subvarieties (we say sometimes that has {\em finite height}): if $\vv V$ is proper than there is a finite subset $X$ of $\{\alg \L_n: n \in \mathbb N\}$ and a finite subset $Y\sse\{\alg \L^\o_m: m \in \mathbb N\}$ with $\vv V = \VV(X \cup Y)$ \cite{AglianoPanti1999}.

Recalling Theorem \ref{kalmanwajsberg}, we have that
\begin{itemize}
	\item $K(\alg \L_n,F)\leq K(\alg \L_m,G)$ if and only if $n|m$ and $\langle F\rangle_m\subseteq G$, where $\langle F\rangle_m$ is the lattice filter generated by $F\subseteq \alg \L_n$, viewed as a subset of $\alg \L_m$. This can be simplified as each lattice filter in $\alg\L_n$ is principal, and therefore if $F=\langle x^r\rangle_n$ and $G=\langle x^s\rangle_m$, then $\alg K_{r,n}=K(\alg\L_n,F)\leq K(\alg\L_m,G)=\alg K_{s,n}$ if and only if there is a $k \in \mathbb N$ with $nk=m$ and $s \ge rk$.
	\item $K(\alg\L_n,F)\leq K(\alg\L^\o_m,G)$ if and only if $n|m$ and $\langle F\rangle_m\subseteq G$.
	\item $K(\alg\L^\o_n,F)\leq K(\alg\L^\o_m,G)$ if and only if $n|m$ and $\langle F\rangle_m\subseteq G$.
\end{itemize}

We stress that we can obtain all the results in Section 6.3 for finitely generated varieties of Wajsberg K-lattices with an even simpler presentation. This is a straightforward, albeit lengthy, exercise and we leave it to the interested reader. As a small example, in Picture \ref{kl4} we include the lattice of subvarieties of $\VV(K(\alg \L_4))$.

\begin{figure}[htbp]
\begin{center}
\begin{tikzpicture}[scale=0.7]
\draw (11.,-4.)-- (11.,-3.);
\draw (11.,-3.)-- (11.,-2.);
\draw (11.,-3.)-- (13.,-2.);
\draw (13.,-2.)-- (15.,-1.);
\draw (13.,-2.)-- (12.,-1.);
\draw (15.,-1.)-- (15.,0.);
\draw (15.,0.)-- (15.,1.);
\draw (15.,0.)-- (13.,-1.);
\draw (13.,-1.)-- (11.,-2.);
\draw (12.,-1.)-- (11.,0.);
\draw (11.,0.)-- (11.,1.);
\draw (11.,1.)-- (12.,0.);
\draw (12.,0.)-- (12.,-1.);
\draw (12.,0.)-- (13.,-1.);
\draw (13.,-1.)-- (13.,-2.);
\draw (15.,-1.)-- (14.,0.);
\draw (15.,0.)-- (14.,1.);
\draw (14.,1.)-- (14.,0.);
\draw (12.,-1.)-- (14.,0.);
\draw (12.,0.)-- (14.,1.);
\draw (14.,1.)-- (13.,2.);
\draw (13.,2.)-- (13.,1.);
\draw (13.,2.)-- (11.,1.);
\draw (11.,0.)-- (13.,1.);
\draw (13.,1.)-- (14.,0.);
\draw (14.,1.)-- (14.,2.);
\draw (13.,2.)-- (13.,3.);
\draw (13.,3.)-- (14.,2.);
\draw (14.,2.)-- (15.,1.);
\draw (13.,1.)-- (12.,2.);
\draw (13.,2.)-- (12.,3.);
\draw (13.,3.)-- (12.,4.);
\draw (12.,2.)-- (12.,3.);
\draw (12.,3.)-- (12.,4.);
\draw (12.,2.)-- (11.,3.);
\draw (12.,3.)-- (11.,4.);
\draw (11.,3.)-- (11.,4.);
\draw (12.,4.)-- (11.,5.);
\draw (11.,4.)-- (11.,5.);
\draw (11.,5.)-- (11.,6.);
\draw [fill] (11.,-3.) circle (2.5pt);\node[left] at (11.,-3.) {\tiny $\VV(\alg K_3)$};
\draw [fill] (11.,-4.) circle (2.5pt);\node[right] at (11.,-4.) {\tiny $\mathsf{T}$};
\draw [fill] (11.,-2.) circle (2.5pt);\node[left] at (11.,-2.) {\tiny $\VV(\alg K_4)=K(\mathsf(BA))$};
\draw [fill] (13.,-2.) circle (2.5pt);\node[right] at (13.,-2.) {\tiny $\VV(\alg K_{0,2})$};
\draw [fill] (15.,-1.) circle (2.5pt);\node[right] at (15.,-1.) {\tiny $\VV(\alg K_{1,2})$};
\draw [fill] (13.,-1.) circle (2.5pt);
\draw [fill] (15.,0.) circle (2.5pt);
\draw [fill] (12.,-1.) circle (2.5pt);\node[left] at (12.,-1.) {\tiny $\VV(\alg K_{0,4})$};
\draw [fill] (11.,0.) circle (2.5pt);\node[left] at (11.,0.) {\tiny $\VV(\alg K_{1,4})$};
\draw [fill] (11.,1.) circle (2.5pt);
\draw [fill] (12.,0.) circle (2.5pt);
\draw [fill] (15.,1.) circle (2.5pt);\node[right] at (15.,1.) {\tiny $\VV(K(\alg \L_2))$};
\draw [fill] (14.,1.) circle (2.5pt);
\draw [fill] (14.,0.) circle (2.5pt);
\draw [fill] (13.,2.) circle (2.5pt);
\draw [fill] (13.,1.) circle (2.5pt);
\draw [fill] (14.,2.) circle (2.5pt);
\draw [fill] (13.,3.) circle (2.5pt);
\draw [fill] (12.,2.) circle (2.5pt);\node[left] at (12.,2.) {\tiny $\VV(\alg K_{2,4})$};
\draw [fill] (12.,3.) circle (2.5pt);
\draw [fill] (12.,4.) circle (2.5pt);
\draw [fill] (11.,3.) circle (2.5pt);\node[left] at (11.,3.) {\tiny $\VV(\alg K_{3,4})$};
\draw [fill] (11.,4.) circle (2.5pt);
\draw [fill] (11.,5.) circle (2.5pt);
\draw [fill] (11.,6.) circle (2.5pt);\node[above] at (11.,6.) {\tiny $\VV(K(\alg \L_4))$};
\end{tikzpicture}
\end{center}
\caption{$\Lambda(\VV(K(\alg \L_4)))$\label{kl4}}
\end{figure}

For the case of non-finitely generated varieties, observe that in $K(\alg\L^\o_n,F)$, if $F$ is a lattice filter properly containing the radical, and if $\theta$ is the congruence in $K(\alg\L^\o_n,F)$ generated by the radical of $\alg\L^\o_n$, then $K(\alg\L^\o_n,F)/\theta \cong \alg K_4$. Therefore in this case the structure of the lattice of subvarieties becomes much more complex.

\subsubsection{Bounded G\"odel and Product K-lattices}

Let's consider the variety $\mathsf{GA}$ of G\"odel algebras; it is obvious that its lattice of subvarieties is identical to the one of G\"odel hoops. If we look at $K(\mathsf{GA})$ things are different though.
Observe first that the following Theorem from \cite{AglianoMarcos2020a} still holds.

\begin{theorem} The nontrivial subalgebras (up to isomorphism) of $\alg K_{n^2}$ are the algebras $\alg K_{m^2}$ and $\alg K_{m^2-1}$ for $m=2,\dots,n$.
\end{theorem}

However, it is no longer true (as in the unbounded case) that $\alg K_{n^2-1}$ has a subalgebra isomorphic with $\alg K_{m^2}$ for $m<n$, neither it is true that $\alg K_{n^2-1}$ has $\alg K_{m^2}$ as a homomorphic image for $m<n$ (this is due to the fact that a homomorphism $f:\alg G_n \to \alg G_m$ satisfies $f(x)=0$ only for $x=0$).
In particular $\Lambda(K(\mathsf{GA}))$ is no longer a chain; with standard calculations one can show that it is the lattice in  Figure \ref{kga}.

\begin{figure}[htbp]
\begin{center}
\begin{tikzpicture}[scale=.8]
\draw (0,0) -- (0,1) -- (-1,1.5) -- (-1,3.5) -- (0,4) -- (0,3) -- (0,2) -- (1,1.5) -- (0,1);
\draw (-1,1.5) -- (0,2);
\draw (-1,3.5) -- (0,4);
\draw[dashed] (-1,3.5) -- (-1,5.5);
\draw[dashed] (0,4) -- (0,6);
\draw (-1,5.5) -- (0,6) -- (0,7) -- (0,8) -- (-1,7.5) -- (-1,5.5);
\draw[dashed] (-1,7.5) -- (-1,8.5);
\draw[dashed] (0,8) -- (0,10);
\draw[fill] (0,0) circle [radius=0.075]; \node[right] at (0,0) {\tiny $\mathsf{T}$};
\draw[fill] (0,1) circle [radius=0.075]; \node[right] at (0,.9) {\tiny $\VV(\alg K_3)$};
\draw[fill] (-1,1.5) circle [radius=0.075]; \node[left] at (-1,1.5) {\tiny $\VV(\alg K_8)$};
\draw[fill] (-1,3.5) circle [radius=0.075]; \node[left] at (-1,3.5) {\tiny $\VV(\alg K_{15})$};
\draw[fill] (-1,5.5) circle [radius=0.075]; \node[left] at (-1,5.5) {\tiny $\VV(\alg K_{(n-1)^2-1})$};
\draw[fill] (-1,7.5) circle [radius=0.075]; \node[left] at (-1,7.5) {\tiny $\VV(\alg K_{n^2-1})$};
\draw[fill] (0,2) circle [radius=0.075];
\draw[fill] (0,4) circle [radius=0.075];
\draw[fill] (0,6) circle [radius=0.075];
\draw[fill] (0,8) circle [radius=0.075];
\draw[fill] (1,1.5) circle [radius=0.075]; \node[right] at (1,1.5) {\tiny $\VV(\alg K_4) = \VV(K(\mathbf 2))$};
\draw[fill] (0,3) circle [radius=0.075]; \node[right] at (0,3) {\tiny $\VV(\alg K_9) = \VV(\alg K(\mathbf 2 \oplus \mathbf 2))$};
\draw[fill] (0,7) circle [radius=0.075]; \node[right] at (0,7) {\tiny $\VV(\alg K_{(n-1)^2}) = \VV(K(\mathbf 2 \oplus\dots\oplus\mathbf 2))$};
\draw[fill] (0,10) circle [radius=0.075]; \node[right] at (0,10) {\tiny $K(\mathsf{GA})$};
\end{tikzpicture}
\end{center}
\caption{$\Lambda(K(\mathsf{GA}))$\label{kga}}
\end{figure}

\medskip

A BL-algebra is a \textbf{product algebra} if it satisfies the equation
\begin{align*}	\neg y \join ((x \imp xy) \imp y) \app 1.\end{align*}
The variety $\mathsf{PA}$ of product algebras has been studied in \cite{CignoliTorrens2000}; joining the results therein with the results about product K-lattices in \cite{AglianoMarcos2020a} we get:

\begin{lemma} The subdirectly irreducible algebras in $K(\mathsf{PA})$ are exactly $\alg K_3$, $\alg K_4$, $K_0(\mathbf 2 \oplus \alg C)$ and $K(\mathbf 2 \oplus \alg C)$ for any totally ordered cancellative hoop $\alg C$.
\end{lemma}

Then it is easily checked that $\Lambda(K(\mathsf{PA}))$ is the lattice in Figure \ref{plattice}.

\begin{figure}[htbp]
\begin{center}
\begin{tikzpicture}
\draw (-1,2) -- (-1,3) -- (0,4) -- (-1,5) -- (-1,6) -- (-1,5) -- (-2,4) -- (-1,3) --(0,4) -- (-1,3);
\draw[fill] (-1,2) circle [radius=0.05];
\draw[fill] (0,4) circle [radius=0.05];
\draw[fill] (-1,5) circle [radius=0.05];
\draw[fill] (-1,6) circle [radius=0.05];
\draw[fill] (-2,4) circle [radius=0.05];
\draw[fill] (-1,3) circle [radius=0.05];
\node[right] at (-1,2) {\footnotesize $\mathsf{T}$};
\node[right] at (0,4) {\footnotesize $\VV(K_0(\alg 2\oplus \alg C_\o))$};
\node[right] at (-1,6) {\footnotesize $K(\mathsf{PA})=\VV(K(\alg 2\oplus \alg C_\o))$};
\node[left] at (-2,4) {\footnotesize $K(\mathsf{BA}) = \VV(\alg K_4)$};
\node[left] at (-1,3) {\footnotesize $\VV(\alg K_3)$};
\end{tikzpicture}
\end{center}
\caption{$\Lambda(K(\mathsf{PA}))$\label{plattice}}
\end{figure}

\subsection{Nilpotent minimum K-lattices}\label{nilpotent}

Let's consider the variety generated by connected and disconnected rotations of G\"odel chains, that is the variety $\mathsf{NM}$ of {\bf nilpotent minimum algebras}. This variety has been completely described in \cite{Gispert2003}; these are the highlights:
\begin{enumerate}
\ib the variety $\mathsf{NM}$ is the subvariety of $\mathsf{BCRL}$ axiomatized by (P), i.e. they are $\mathsf{MTL}$ algebras, $\neg\neg x \app x$ (which gives involution) and
$$
(xy \imp 0) \join ((x \meet y) \imp xy) \app 1;
$$
\ib the finite {$\mathsf{NM}$}-chains are exactly connected or disconnected rotations of G\"odel chains; more precisely if $\alg {N}_k$ denotes the $k$-element nilpotent minimum chain, then we have that $\alg{N}_{2n}$ is the disconnected rotation of the G\"odel chain $\alg G_n$  and $\alg {N}_{2n+1}$ is the connected rotation of $\alg G_n$ (of course $\alg G_0$ is the trivial algebra). Clearly $\alg {N}_2 = \mathbf 2$ and $\alg {N}_3 = \alg \L_2$;
\ib we define $\alg N_\o$ to be the nilpotent minimum algebra whose universe is the real interval $[0,1]$ the order is the natural order and the operations are:
$$
xy = \left\{
       \begin{array}{ll}
         0, & \hbox{if $y \le 1-x$;} \\
         \min(x,y), & \hbox{otherwise.}
       \end{array}
     \right.
\qquad
x \imp y \left\{
           \begin{array}{ll}
             1, & \hbox{if $x \le y$;} \\
             \max((1-x),y), & \hbox{otherwise.}
           \end{array}
         \right.
$$
It is clear that $\neg x = 1-x$ and hence $\neg(1/2)=1/2$. It follows that $[0,1]\setminus \{1/2\}$ is the universe of a subalgebra of $\alg N_\o$ that we denote by $\alg N^*_\o$;
\ib $\alg{N}_{2k}, \alg{N}_{2k+1} \le \alg{N}_{2n+1}$ for $k \le n$; $\alg{N}_{2k} \le \alg{N}_{2n}$ for $k \le n$; $\alg N_{2n} \le \alg N^*_\o$ for all $n$; $\alg N_{2n+1} \le \alg N_\o$ for all $n$;
\ib $\HH(\alg N_{2n+1}) = \{\alg N_k: k \le 2n+1\}$, $\HH(\alg N_{2n}) =\{\alg N_{2k} : k \le n\}$;
\ib the variety $\mathsf{NM}$ is generated by all the chains $\{\alg{N}_{2n+1}: n\in \mathbb N\}$ or by any infinite chain having an element $a$ such that $\neg a=a$; hence $\mathsf{NM} = \VV(\alg N_\o)$;
\ib every subvariety of $\mathsf{NM}$ is generated by its finite algebras, so it has the finite model property and each splitting algebra is finite;
\ib $\mathsf{NM}^* = \VV(\alg N^*_\o) = \VV(\{\alg N_{2k}: k \in \mathbb N\})$ is a proper subvariety of $\mathsf{NM}$.
\end{enumerate}
Now putting together all we know, we observe that:

\begin{lemma}\label{hspu} For every $n \in \mathbb N$,  $\HH\SU\PP_u(K(\alg {N}_{2n+1})) \sse \SU K(\{\alg {N}_k : k \le 2n+1\})$ and $\HH\SU\PP_u(K(\alg {N}_{2n})) \sse \SU K(\{\alg {N}_{2k} : k \le n\})$.
\end{lemma}
\begin{proof} We compute:
\begin{align*}
\HH\SU\PP_u(K(\alg {N}_{2n+1})) &\sse \SU K(\HH\SU\PP_u(\alg{N}_{2n+1})) = \SU K(\HH\SU(\alg{N}_{2n+1}))\\
&= \SU K(\HH\SU(\alg G_n^{\d_2}))= \SU K(\HH\SU(\alg G_n)^{\d_2})\\
&= \SU K(\HH\SU(\{\alg G_k: k \le 2n+1\}^{\d_2})) \\
&= \SU K(\{\alg {N}_k : k \le 2n+1\});
\end{align*}
the second point follows by a similar argument.
\end{proof}

It is easy to see that the non-admissible subalgebras of $K(\alg{N}_{2n})$ are exactly the  admissible subalgebras of $K(\alg{N}_{2k})$ with $k \le n$ and the non admissible subalgebras of $K(\alg{N}_{2n+1})$ are exactly the admissible subalgebras of $K(\alg{N}_{2k})$ and $K(\alg{N}_{2k+1})$ with $k\le n$. Since $\mathsf{NM}$ is involutive the admissible subalgebras can be computed using Theorem \ref{kalmanwajsberg}; it follows that
that $K(\alg N_{k})$ has $k$ admissible subalgebras and they can be found by computing $x \oplus y = (x\imp 0) \imp y$ on each $\alg N_k$. As an example in Figure \ref{n5} we draw the proper admissible subalgebras of $K(\alg N_5)$, where $N_5 = \{0<\neg a< b=\neg b< a<1\}$.

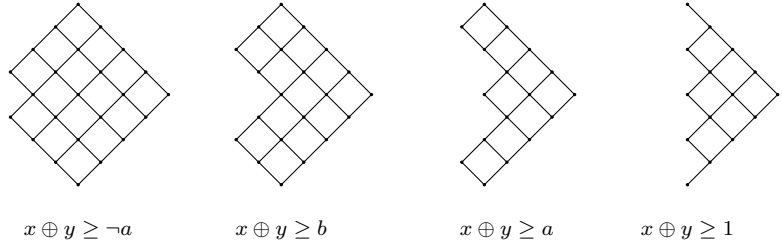
\begin{figure}[htbp]
\begin{center}
\begin{tikzpicture}[scale=.3]
\draw (0,2) -- (4,6) -- (0,10);
\draw (0,2) -- (-3,5);
\draw (0,10) --(-3,7);
\draw (1,3) -- (-3,7);
\draw (2,4) -- (-2,8);
\draw (3,5) -- (-1,9);
\draw (-1,3) -- (3,7);
\draw (-2,4) -- (2,8);
\draw (-3,5) -- (1,9);
\draw[fill] (0,2) circle [radius=0.05];
\draw[fill] (0,4) circle [radius=0.05];
\draw[fill] (0,6) circle [radius=0.05];
\draw[fill] (0,8) circle [radius=0.05];
\draw[fill] (0,10) circle [radius=0.05];
\draw[fill] (-1,3) circle [radius=0.05];
\draw[fill] (-1,5) circle [radius=0.05];
\draw[fill] (-1,7) circle [radius=0.05];
\draw[fill] (-1,9) circle [radius=0.05];
\draw[fill] (-2,4) circle [radius=0.05];
\draw[fill] (-2,6) circle [radius=0.05];
\draw[fill] (-2,8) circle [radius=0.05];
\draw[fill] (-3,5) circle [radius=0.05];
\draw[fill] (-3,7) circle [radius=0.05];
\draw[fill] (1,3) circle [radius=0.05];
\draw[fill] (1,5) circle [radius=0.05];
\draw[fill] (1,7) circle [radius=0.05];
\draw[fill] (1,9) circle [radius=0.05];
\draw[fill] (2,4) circle [radius=0.05];
\draw[fill] (2,6) circle [radius=0.05];
\draw[fill] (2,8) circle [radius=0.05];
\draw[fill] (3,5) circle [radius=0.05];
\draw[fill] (3,7) circle [radius=0.05];
\draw[fill] (4,6) circle [radius=0.05];
\node at (0,0) {\footnotesize $x \oplus y \ge \neg a$};
%**********************
\draw (7,4) -- (9,2) -- (13,6) -- (9,10) -- (7,8);
\draw (8,3) -- (12,7);
\draw (7,4) -- (11,8);
\draw(8,7) -- (10,9);
\draw (10,3) -- (8,5);
\draw (11,4) -- (7,8);
\draw (12,5) -- (8,9);
\draw[fill] (9,2) circle [radius=0.05];
\draw[fill] (9,4) circle [radius=0.05];
\draw[fill] (9,6) circle [radius=0.05];
\draw[fill] (9,8) circle [radius=0.05];
\draw[fill] (9,10) circle [radius=0.05];
\draw[fill] (8,3) circle [radius=0.05];
\draw[fill] (8,5) circle [radius=0.05];
\draw[fill] (8,7) circle [radius=0.05];
\draw[fill] (8,9) circle [radius=0.05];
\draw[fill] (7,4) circle [radius=0.05];
\draw[fill] (7,8) circle [radius=0.05];
\draw[fill] (10,3) circle [radius=0.05];
\draw[fill] (10,5) circle [radius=0.05];
\draw[fill] (10,7) circle [radius=0.05];
\draw[fill] (10,9) circle [radius=0.05];
\draw[fill] (11,4) circle [radius=0.05];
\draw[fill] (11,6) circle [radius=0.05];
\draw[fill] (11,8) circle [radius=0.05];
\draw[fill] (12,5) circle [radius=0.05];
\draw[fill] (12,7) circle [radius=0.05];
\draw[fill] (13,6) circle [radius=0.05];
\node at (9,0) {\footnotesize $x \oplus y \ge b$};
%******************************
\draw (17,3) -- (18,2) -- (22,6) -- (18,10) -- (17,9);
\draw (19,3) -- (18,4);
\draw (20,4) -- (19,5);
\draw (17,3) -- (21,7);
\draw (17,9) -- (21,5);
\draw (20,8) -- (19,7);
\draw (19,9) -- (18,8);
\draw (19,5) -- (18,6) -- (19,7);
\draw[fill] (18,2) circle [radius=0.05];
\draw[fill] (18,4) circle [radius=0.05];
\draw[fill] (18,6) circle [radius=0.05];
\draw[fill] (18,8) circle [radius=0.05];
\draw[fill] (18,10) circle [radius=0.05];
\draw[fill] (17,3) circle [radius=0.05];
\draw[fill] (17,9) circle [radius=0.05];
\draw[fill] (19,3) circle [radius=0.05];
\draw[fill] (19,5) circle [radius=0.05];
\draw[fill] (19,7) circle [radius=0.05];
\draw[fill] (19,9) circle [radius=0.05];
\draw[fill] (20,4) circle [radius=0.05];
\draw[fill] (20,6) circle [radius=0.05];
\draw[fill] (20,8) circle [radius=0.05];
\draw[fill] (21,5) circle [radius=0.05];
\draw[fill] (21,7) circle [radius=0.05];
\draw[fill] (22,6) circle [radius=0.05];
\node at (19,0) {\footnotesize $x \oplus y \ge a$};
%******************************
\draw (27,2) -- (31,6) -- (27,10);
\draw (28,9) -- (27,8);
\draw (29,8) -- (28,7);
\draw (30,7) -- (27,4);
\draw (28,3) -- (27,4);
\draw (29,4) -- (28,5);
\draw (30,5) -- (27,8);
\draw (28,5) -- (27,6) -- (28,7);
\draw[fill] (27,2) circle [radius=0.05];
\draw[fill] (27,4) circle [radius=0.05];
\draw[fill] (27,6) circle [radius=0.05];
\draw[fill] (27,8) circle [radius=0.05];
\draw[fill] (27,10) circle [radius=0.05];
\draw[fill] (28,3) circle [radius=0.05];
\draw[fill] (28,5) circle [radius=0.05];
\draw[fill] (28,7) circle [radius=0.05];
\draw[fill] (28,9) circle [radius=0.05];
\draw[fill] (29,4) circle [radius=0.05];
\draw[fill] (29,6) circle [radius=0.05];
\draw[fill] (29,8) circle [radius=0.05];
\draw[fill] (30,5) circle [radius=0.05];
\draw[fill] (30,7) circle [radius=0.05];
\draw[fill] (31,6) circle [radius=0.05];
\node at (27,0) {\footnotesize $x \oplus y \ge 1$};
\end{tikzpicture}
\end{center}
\caption{The admissible subalgebras of $K(\alg N_5)$\label{n5}}
\end{figure}

In order to describe the lattice of subvarieties of $K(\mathsf{NM})$, we first observe that $\alg \L_2$ and $\alg N_4$ are the only rigid algebras in $\mathsf{NM}$, therefore each one generates a cover of $\VV(\alg K_3)$ different from $K(\mathsf{BA})$.

Then we observe that there exist retractions $\gamma_{2n}:\alg N_{2n}\to \alg 2$ for each $k$. Thus if we consider $K(\alg N_{2n},F)$ for some lattice filter $F$ containing $x$ such that $\gamma_{2n}(x)=0$ (that is lattice filters strictly greater than the radical), then we will have $\alg K_4 \in \HH(K(\alg N_{2n},F))$ by considering the morphism $f(a,b)=(\gamma_{2n}(a),\gamma_{2n}(b))$. It is also easy to see that if the filter $F$ is contained in the radical (i.e. $\gamma_{2n}(x)=1$ for each $x\in F$), then $\alg K_4 \not\in \HH(K(\alg N_{2n},F))$.
Moreover, if we consider a morphism $\gamma_{2n,2k}:\alg N_{2n}\to \alg N_{2k}$, then $\alg K(\alg N_{2k}) \in \HH(K(\alg N_{2n},F))$ if and only if the lattice filter $F$ contains an element $x$ such that $\gamma_{2n,2k}(x)=0$.

The case of $\alg N_{2n+1}$ is similar, if we consider a morphism $\gamma_{2n+1,2k+1}:\alg N_{2n+1}\to \alg N_{2k+1}$, then $\alg K(\alg N_{2k+1}) \in \HH(K(\alg N_{2n+1},F))$ if and only if the lattice filter $F$ contains an element $x$ such that $\gamma_{2n+1,2k+1}(x)=0$.
In Figure \ref{N5lattice} we describe the lattice of subvarieties of $\VV(K(\alg N_5))$.

To simplify notation, we name $K_0(\alg \L_2)\leq K_1(\alg \L_2)\leq K(\alg \L_2)$, $K_0(\alg N_4)\leq K_1(\alg N_4)\leq K_2(\alg N_4)\leq K(\alg N_4)$ and $K_0(\alg N_5)\leq K_1(\alg N_5)\leq K_2(\alg N_5)\leq K_3(\alg N_5)\leq K(\alg N_5)$ the admissible subalgebras of $\alg \L_2$, $\alg N_4$ and $\alg N_5$, respectively. From the previous observations, we recall that

\begin{itemize}
	\item $\alg K_4 \leq K(\alg\L_2)$ and $K(\alg N_4) \leq K(\alg N_5)$,
	\item $K_0(\alg \L_2) \leq K_0(\alg N_5)$ and $K_0(\alg N_4) \leq K_0(\alg N_5)$,
	\item $K_1(\alg \L_2) \leq K_2(\alg N_5)$,
	\item $K_1(\alg N_4) \leq K_1(\alg N_5)$,
	\item $K_2(\alg N_4) \leq K_3(\alg N_5)$,
	\item $\alg K_4 \in \HH(K_2(\alg N_4))$,
	\item $K(\alg\L_2)\in\HH(K_3(\alg N_5))$.
\end{itemize}

With this information, the reader can verify that the lattice of subvarieties of $\VV(K(\alg N_5))$ is effectively the one drawn in Figure \ref{N5lattice}.

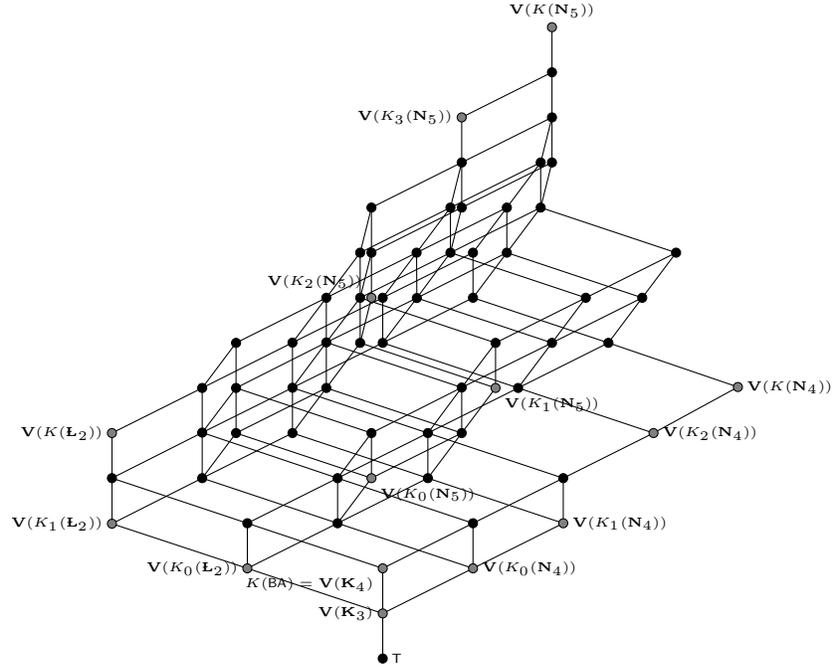
\begin{figure}[htbp]
\begin{center}
\begin{tikzpicture}[scale=.03]
\draw (180.,-140.)-- (180.,-120.);
\draw (180.,-140.)-- (120.,-120.);
\draw (180.,-140.)-- (220.,-120.);
\draw (120.,-120.)-- (120.,-100.);
\draw (180.,-120.)-- (120.,-100.);
\draw (180.,-120.)-- (220.,-100.);
\draw (220.,-120.)-- (220.,-100.);
\draw (120.,-120.)-- (160.,-100.);
\draw (220.,-120.)-- (160.,-100.);
\draw (120.,-100.)-- (160.,-80.);
\draw (160.,-100.)-- (160.,-80.);
\draw (220.,-100.)-- (160.,-80.);
\draw (120.,-120.)-- (60.,-100.2);
\draw (220.,-120.)-- (260.,-100.);
\draw (60.,-100.2)-- (60.,-80.);
\draw (120.,-100.)-- (60.,-80.);
\draw (60.,-100.2)-- (100.,-80.);
\draw (160.,-100.)-- (100.,-80.);
\draw (160.,-100.)-- (200.,-80.);
\draw (260.,-100.)-- (200.,-80.);
\draw (220.,-100.)-- (260.,-80.);
\draw (260.,-100.)-- (260.,-80.);
\draw (60.,-80.)-- (100.,-59.8);
\draw (100.,-80.)-- (100.,-59.8);
\draw (200.,-80.)-- (200.,-60.);
\draw (260.,-80.)-- (200.,-60.);
\draw (160.,-80.)-- (100.,-59.8);
\draw (160.,-80.)-- (200.,-60.);
\draw (100.,-59.8)-- (140.,-39.9);
\draw (140.,-39.9)-- (200.,-60.);
\draw (140.,-60.)-- (140.,-39.9);
\draw (100.,-80.)-- (140.,-60.);
\draw (200.,-80.)-- (140.,-60.);
\draw (160.,-100.)-- (175.,-80.);
\draw (175.,-80.)-- (115.,-60.);
\draw (175.,-80.)-- (215.,-60.);
\draw (175.,-80.)-- (175.,-60.);
\draw (160.,-80.)-- (175.,-60.);
\draw (100.,-80.)-- (115.,-60.);
\draw (200.,-80.)-- (215.,-60.);
\draw (215.,-60.)-- (215.,-39.80205530786387);
\draw (175.,-60.)-- (215.,-39.80205530786387);
\draw (175.,-60.)-- (115.,-40.);
\draw (115.,-60.)-- (115.,-40.);
\draw (100.,-59.8)-- (115.,-40.);
\draw (200.,-60.)-- (215.,-39.80205530786387);
\draw (115.,-40.)-- (155.,-19.802055307863867);
\draw (140.,-39.9)-- (155.,-19.802055307863867);
\draw (215.,-39.80205530786387)-- (155.,-19.802055307863867);
\draw (115.,-60.)-- (155.,-40.);
\draw (215.,-60.)-- (155.,-40.);
\draw (140.,-60.)-- (155.,-40.);
\draw (155.,-40.)-- (155.,-19.802055307863867);
\draw (215.,-60.)-- (230.,-40.);
\draw (155.,-40.)-- (170.,-20.);
\draw (230.,-40.)-- (170.,-20.);
\draw (215.,-39.80205530786387)-- (230.,-20.);
\draw (230.,-40.)-- (230.,-20.);
\draw (230.,-20.)-- (170.,0.);
\draw (170.,-20.)-- (170.,0.);
\draw (155.,-19.802055307863867)-- (170.,0.);
\draw (60.,-80.)-- (60.,-60.);
\draw (60.,-60.)-- (100.,-40.);
\draw (100.,-59.8)-- (100.,-40.);
\draw (155.,0.)-- (115.,-20.);
\draw (140.,-19.9)-- (155.,0.);
\draw (100.,-40.)-- (115.,-20.);
\draw (100.,-40.)-- (140.,-19.9);
\draw (115.,-40.)-- (115.,-20.);
\draw (155.,-19.802055307863867)-- (155.,0.);
\draw (155.,0.)-- (170.,20.);
\draw (170.,0.)-- (170.,20.);
\draw (170.,-20.)-- (175.,0.);
\draw (170.,0.)-- (175.,20.);
\draw (175.,0.)-- (175.,20.);
\draw (175.,20.)-- (175.,40.);
\draw (170.,20.)-- (175.,40.);
\draw (260.,-80.)-- (300.,-60.);
\draw (200.,-60.)-- (240.,-40.);
\draw (300.,-60.)-- (240.,-40.);
\draw (240.,-40.)-- (180.,-20.);
\draw (240.,-40.)-- (255.,-20.);
\draw (180.,-20.)-- (180.,0.);
\draw (255.,-20.)-- (270.,0.);
\draw (255.,-20.)-- (195.1144425449063,-0.11444254490632133);
\draw (180.,-20.)-- (195.1144425449063,-0.11444254490632133);
\draw (140.,-39.9)-- (180.,-20.);
\draw (215.,-39.80205530786387)-- (255.,-20.);
\draw (140.,-19.9)-- (180.,0.);
\draw (155.,-19.802055307863867)-- (195.1144425449063,-0.11444254490632133);
\draw (230.,-20.)-- (270.,0.);
\draw (180.,0.)-- (195.,20.);
\draw (195.1144425449063,-0.11444254490632133)-- (195.,20.);
\draw (155.,0.)-- (195.,20.);
\draw (195.1144425449063,-0.11444254490632133)-- (210.,20.);
\draw (270.,0.)-- (210.,20.);
\draw (210.,20.)-- (215.,40.);
\draw (170.,0.)-- (210.,20.);
\draw (175.,20.)-- (215.,40.);
\draw (170.,20.)-- (210.,40.);
\draw (195.,20.)-- (210.,40.);
\draw (210.,20.)-- (210.,40.);
\draw (175.,40.)-- (215.,60.);
\draw (215.,40.)-- (215.,60.);
\draw (210.,40.)-- (215.,60.);
\draw (215.,60.)-- (215.,80.);
\draw (300.,-60.)-- (337.39363185266586,-39.58086655333234);
\draw (337.39363185266586,-39.58086655333234)-- (280.,-20.);
\draw (280.,-20.)-- (220.,0.);
\draw (220.,0.)-- (220.,20.);
\draw (280.,-20.)-- (295.,0.);
\draw (220.,0.)-- (235.,20.);
\draw (295.,0.)-- (235.,20.);
\draw (295.,0.)-- (310.,20.);
\draw (310.,20.)-- (250.,40.);
\draw (235.,20.)-- (250.,40.);
\draw (220.,20.)-- (235.,40.);
\draw (235.,20.)-- (235.,40.);
\draw (235.,40.)-- (250.,60.);
\draw (250.,40.)-- (250.,60.);
\draw (250.,40.)-- (255.,60.);
\draw (250.,60.)-- (255.,80.);
\draw (255.,60.)-- (255.,80.);
\draw (255.,80.)-- (255.,100.);
\draw (240.,-40.)-- (280.,-20.);
\draw (255.,-20.)-- (295.,0.);
\draw (180.,-20.)-- (220.,0.);
\draw (270.,0.)-- (310.,20.);
\draw (195.1144425449063,-0.11444254490632133)-- (235.,20.);
\draw (180.,0.)-- (220.,20.);
\draw (210.,20.)-- (250.,40.);
\draw (195.,20.)-- (235.,40.);
\draw (215.,40.)-- (255.,60.);
\draw (210.,40.)-- (250.,60.);
\draw (215.,60.)-- (255.,80.);
\draw (215.,80.)-- (255.,100.);
\draw (255.,100.)-- (255.,120.);
\draw (180.,-160.)-- (180.,-140.);
\draw (140.,-39.9)-- (140.,-19.9);
\draw [fill=gray] (180.,-140.) circle [radius=2];
\node[left] at (180.,-140.) {\tiny $\VV(\alg K_3)$};
\draw [fill=gray] (120.,-120.) circle [radius=2];
\node[left] at (120.,-120.) {\tiny $\VV(K_0(\alg \L_2))$};
\draw [fill=gray] (180.,-120.) circle [radius=2];
\node[below left] at (180.,-120.) {\tiny $K(\mathsf{BA}) = \VV(\alg K_4)$};
\draw [fill=gray] (220.,-120.) circle [radius=2];
\node[right] at (220.,-120.) {\tiny $\VV(K_0(\alg N_4))$};
\draw [fill=black] (160.,-100.) circle [radius=2];
\draw [fill=black] (120.,-100.) circle [radius=2];
\draw [fill=black] (220.,-100.) circle [radius=2];
\draw [fill=black] (160.,-80.) circle [radius=2];
\draw [fill=gray] (60.,-100.2) circle [radius=2];
\node[left] at (60.,-100.2) {\tiny $\VV(K_1(\alg \L_2))$};
\draw [fill=gray] (260.,-100.) circle [radius=2];
\node[right] at (260.,-100.) {\tiny $\VV(K_1(\alg N_4))$};
\draw [fill=black] (60.,-80.) circle [radius=2];
\draw [fill=black] (100.,-80.) circle [radius=2];
\draw [fill=black] (100.,-59.8) circle [radius=2];
\draw [fill=black] (140.,-60.) circle [radius=2];
\draw [fill=black] (200.,-60.) circle [radius=2];
\draw [fill=black] (200.,-80.) circle [radius=2];
\draw [fill=black] (260.,-80.) circle [radius=2];
\draw [fill=black] (140.,-39.9) circle [radius=2];
\draw [fill=gray] (175.,-80.) circle [radius=2];
\node[below right] at (175.,-80.) {\tiny $\VV(K_0(\alg N_5))$};
\draw [fill=black] (115.,-60.) circle [radius=2];
\draw [fill=black] (215.,-60.) circle [radius=2];
\draw [fill=black] (175.,-60.) circle [radius=2];
\draw [fill=black] (115.,-40.) circle [radius=2];
\draw [fill=black] (215.,-39.80205530786387) circle [radius=2];
\draw [fill=black] (155.,-19.802055307863867) circle [radius=2];
\draw [fill=black] (155.,-40.) circle [radius=2];
\draw [fill=gray] (230.,-40.) circle [radius=2];
\node[below right] at (230.,-40.) {\tiny $\VV(K_1(\alg N_5))$};
\draw [fill=black] (170.,0.) circle [radius=2];
\draw [fill=black] (170.,-20.) circle [radius=2];
\draw [fill=black] (230.,-20.) circle [radius=2];
\draw [fill=gray] (60.,-60.) circle [radius=2];
\node[left] at (60.,-60.) {\tiny $\VV(K(\alg \L_2))$};
\draw [fill=black] (100.,-40.) circle [radius=2];
\draw [fill=black] (115.,-20.) circle [radius=2];
\draw [fill=black] (140.,-19.9) circle [radius=2];
\draw [fill=black] (155.,0.) circle [radius=2];
\draw [fill=black] (170.,20.) circle [radius=2];
\draw [fill=gray] (175.,0.) circle [radius=2];
\node[above left] at (175.,0.) {\tiny $\VV(K_2(\alg N_5))$};
\draw [fill=black] (175.,20.) circle [radius=2];
\draw [fill=black] (175.,40.) circle [radius=2];
\draw [fill=gray] (300.,-60.) circle [radius=2];
\node[right] at (300.,-60.) {\tiny $\VV(K_2(\alg N_4))$};
\draw [fill=black] (180.,0.) circle [radius=2];
\draw [fill=black] (180.,-20.) circle [radius=2];
\draw [fill=black] (195.1144425449063,-0.11444254490632133) circle [radius=2];
\draw [fill=black] (255.,-20.) circle [radius=2];
\draw [fill=black] (240.,-40.) circle [radius=2];
\draw [fill=black] (270.,0.) circle [radius=2];
\draw [fill=black] (215.,40.) circle [radius=2];
\draw [fill=black] (195.,20.) circle [radius=2];
\draw [fill=black] (210.,20.) circle [radius=2];
\draw [fill=black] (210.,40.) circle [radius=2];
\draw [fill=black] (215.,60.) circle [radius=2];
\draw [fill=gray] (215.,80.) circle [radius=2];
\node[left] at (215.,80.) {\tiny $\VV(K_3(\alg N_5))$};
\draw [fill=gray] (337.39363185266586,-39.58086655333234) circle [radius=2];
\node[right] at (337.39363185266586,-39.58086655333234) {\tiny $\VV(K(\alg N_4))$};
\draw [fill=black] (280.,-20.) circle [radius=2];
\draw [fill=black] (295.,0.) circle [radius=2];
\draw [fill=black] (310.,20.) circle [radius=2];
\draw [fill=black] (220.,0.) circle [radius=2];
\draw [fill=black] (255.,60.) circle [radius=2];
\draw [fill=black] (220.,20.) circle [radius=2];
\draw [fill=black] (235.,20.) circle [radius=2];
\draw [fill=black] (250.,40.) circle [radius=2];
\draw [fill=black] (255.,100.) circle [radius=2];
\draw [fill=black] (235.,40.) circle [radius=2];
\draw [fill=black] (255.,80.) circle [radius=2];
\draw [fill=black] (250.,60.) circle [radius=2];
\draw [fill=gray] (255.,120.) circle [radius=2];
\node[above] at (255.,120.) {\tiny $\VV(K(\alg N_5))$};
\draw [fill=black] (180.,-160.) circle [radius=2];
\node[right] at (180.,-160.) {\tiny $\mathsf{T}$};
\end{tikzpicture}
\end{center}
\caption{$\Lambda(\VV(K(N_5)))$\label{N5lattice}}
\end{figure}

\subsection{Drastic product K-lattices}

A {\bf drastic product chain}, briefly a DP-chain, is  chain in $\mathsf{BCRL}$ where the monoid operation (which determines the entire structure) is defined as
$$
ab = \left\{
       \begin{array}{ll}
         0, & \hbox{if $a,b \ne 1$;} \\
         a \meet b, & \hbox{otherwise.}
       \end{array}
     \right.
$$
We denote by $\alg {DP}_n$ the DP-chain with $n$-elements; note that each DP-chain is simple and $\alg{DP}_n$ has subalgebras isomorphic with $\alg{DP}_k$ for all $k\le n$. In particular $\alg{DP}_3 \cong \alg \L_2$.
DP-chains have been considered (under a different name) by C. Noguera in his PhD Thesis \cite{Noguera2007} and investigated in \cite{AguzzoliBianchiValota2014}. In details:
\begin{enumerate}
\ib the variety $\mathsf {DP}$ generated by all DP-chains is axiomatized with respect to $\mathsf{MTL}$ by $x \join \neg(x^2) \app 1$;
\ib $\mathsf{DP}$ is generated by any infinite chain o by any infinite set of finite chains;
\ib hence any proper subvariety of $\mathsf{DP}$ is generated by a single finite chain and thus $\Lambda(\mathsf{DP})$ is a chain of order $\o+1$.
\end{enumerate}
We remark that any DP-chain has a coatom $a_1$; in fact if $\alg D$ is an infinite DP-chain and $u,v \in D\setminus\{0,1\}$ then $\neg u = \neg v$ and we can take $a$ to be that common value; moreover in a DP-chain with coatom $a_1$
$$
x \imp y =\left\{
            \begin{array}{ll}
              1, & \hbox{if $x \le y$;} \\
              a_1, & \hbox{if $1 >x > y$;} \\
              y, & \hbox{if $x=1$.}
            \end{array}
          \right.
$$
The main difference between this example and the previous ones is that in general the admissible subalgebras of $K(\alg{DP}_n)$ do not form a chain; as a matter of fact they form a chain if and only if $n \le 4$ and this
will be clear after our description of the lattice of admissible subalgebras of $K(\alg {DP}_n)$.

To investigate the admissible subalgebras of the Kalman product of each DP-chain $\alg{DP}_n$, we assume $n\geq 4$ (as $\alg{DP}_2 \cong \alg 2$ and $\alg{DP}_3 \cong \alg \L_2$, those cases have already been considered), and let
$$
DP_n =\{0=a_{n-1} \prec a_{n-2} \prec \dots \prec a_1 \prec 1\}.
$$
Define
\begin{align*} K_n^\emptyset = \{(x,1),(1,x),(x,a_1),(a_1,x):x\in DP_n\}. \end{align*}

\begin{lemma}For $n\geq 4$, $\alg K_n^\emptyset$ is the smallest admissible subalgebra of $K(\alg{DP}_n)$. Moreover, every sublattice of $K(\alg{DP}_n)$ closed under $\sim$ and containing $K_n^\emptyset$ is the universe of an admissible subalgebra of $K(\alg{DP}_n)$.\end{lemma}
\begin{proof}
To show that each admissible subalgebra of $\alg{DP}_n$ contains $K_n^\emptyset$, observe that, for $i=1,\ldots,n-2$,
\begin{align*} (a_i,1)\cdot (1,0) = (a_i,\neg a_i) = (a_i,a_1),\end{align*}
and then if $i=2,\ldots,n-2$,
\begin{align*} (a_1,a_i)\cdot (a_1,a_1) = (0,a_1) = (a_{n-1},a_1),\end{align*}
so the elements $(a_i,a_1),(a_1,a_i)$ belong to each admissible subalgebra.

As $K_n^\emptyset$ is clearly the universe of a sublattice closed under $\sim$, we only need to show that any sublattice of $K(\alg{DP}_n)$ closed under $\sim$ and containing $K_n^\emptyset$ is closed under the product. To show this, observe that for $i,j,k,l = 1,\ldots,n-1$,
\begin{align*} (a_i,a_j)\cdot (a_k,a_l) &= (0,a_i\imp a_l \meet a_k\imp a_j),\\
 (a_i,a_j)\cdot (a_k,1) &= (0,a_k\imp a_j),\\
 (a_i,a_j)\cdot (1,a_k) &= (a_i,a_j),
\end{align*}
and each implication is either $a_1$ or $1$.
\end{proof}

To describe all the subalgebras, consider the lattice with universe
\begin{align*} X_n = \{(a_i,a_j): i\geq j>1\}\end{align*}
where the ordering is given coordinate-wise.

\begin{figure}[htbp]
\begin{center}
\begin{tikzpicture}[scale=.4]
\draw (13.,-7.)-- (14.,-8.);
\draw (14.,0.)-- (13.,-1.);
\draw (14.,-6.)-- (15.,-7.);
\draw (14.,-2.)-- (15.,-1.);
\draw (14.,-2.)-- (13.,-3.);
\draw (14.,-6.)-- (13.,-5.);
\draw (13.,-9.)-- (15.,-7.);
\draw (13.,-7.)-- (14.,-6.);
\draw (13.,1.)-- (15.,-1.);
\draw (13.,-1.)-- (14.,-2.);
\draw (17.,-3.)-- (16.,-4.);
\draw (16.,-4.)-- (17.,-5.);
\draw (17.,-5.)-- (18.,-4.);
\draw (18.,-4.)-- (17.,-3.);
\draw [dash pattern=on 2pt off 2pt] (17.,-3.)-- (15.,-1.);
\draw [dash pattern=on 2pt off 2pt] (16.,-4.)-- (14.,-2.);
\draw [dash pattern=on 2pt off 2pt] (16.,-4.)-- (14.,-6.);
\draw [dash pattern=on 2pt off 2pt] (15.,-7.)-- (17.,-5.);
\draw [dash pattern=on 2pt off 2pt] (13.,-5.)-- (13.5,-4.5);
\draw [dash pattern=on 2pt off 2pt] (13.,-3.)-- (13.5,-3.5);
\draw [fill] (18.,-4.) circle (2.5pt);
\node[right] at (18.,-4.) {$(a_{n-1},a_2)=(0,a_2)$};
\draw [fill] (13.,-9.) circle (2.5pt);
\node[below] at (13.,-9.) {$(a_{n-1},a_{n-1})=(0,0)$};
\draw [fill] (13.,1.) circle (2.5pt);
\node[above] at (13.,1.) {$(a_{2},a_{2})$};
\draw [fill] (14.,0.) circle (2.5pt);
\draw [fill] (13.,-1.) circle (2.5pt);
\draw [fill] (14.,-2.) circle (2.5pt);
\draw [fill] (15.,-1.) circle (2.5pt);
\draw [fill] (13.,-7.) circle (2.5pt);
\draw [fill] (14.,-8.) circle (2.5pt);
\draw [fill] (14.,-6.) circle (2.5pt);
\draw [fill] (15.,-7.) circle (2.5pt);
\draw [fill] (17.,-5.) circle (2.5pt);
\draw [fill] (17.,-3.) circle (2.5pt);
\draw [fill] (16.,-4.) circle (2.5pt);
\draw [fill] (13.,-3.) circle (2.5pt);
\draw [fill] (13.,-5.) circle (2.5pt);
\end{tikzpicture}
\end{center}
\caption{The lattice $X_n$.}
\end{figure}
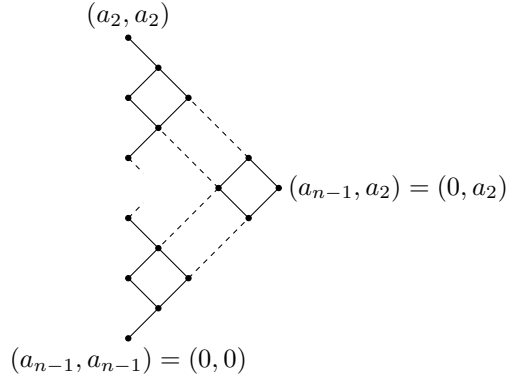

Now, for each up-set $U$ of $X_n$ consider
\begin{align*} K_n^U = K_n^\emptyset \cup \{(x,y): (x,y)\in U \text{ or } (y,x)\in U\}. \end{align*}

From this definition and the previous results, the following Theorem is immediate.

\begin{theorem}For $n\geq 4$, the admissible subalgebras of $K(\alg{DP}_n)$ are $\alg K_n^U$, for each $U$ up-set of $X_n$. Moreover, $\alg K_n^U \subset \alg K_n^W$ if and only if $U\subset W$.\end{theorem}

We also note:

\begin{corollary} The lattice of admissible subalgebras of $K(\alg {DP}_n)$ is a chain if and only if $n \le 4$.
\end{corollary}

Figures \ref{d4} and \ref{d5} show all the admissible subalgebras of $K(\alg {DP}_4)$ and $K(\alg {DP}_5)$.

\medskip

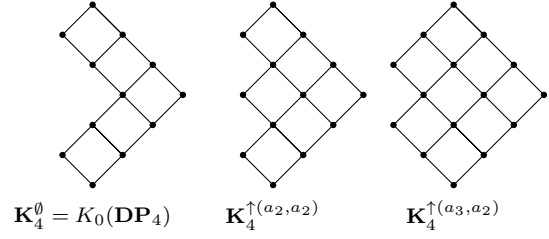
\begin{figure}[htbp]
\begin{center}
\begin{tikzpicture}[scale=.4]
\draw (6,2) -- (7,3) -- (8,4) -- (9,5) -- (8,6) -- (7,5) -- (8,4);
\draw (8,6) -- (7,7) -- (6,8);
\draw  (6,2) -- (5,3) -- (6,4) -- (7,3) -- (6,4) --(7,5) -- (6,6) -- (5,7) -- (6,8) -- (7,7) -- (6,6);
\draw[fill] (5,3) circle (2.5pt);
\draw[fill] (5,7) circle (2.5pt);
\draw[fill] (6,2) circle (2.5pt);
\draw[fill] (6,4) circle (2.5pt);
\draw[fill] (6,6) circle (2.5pt);
\draw[fill] (6,8) circle (2.5pt);
\draw[fill] (7,3) circle (2.5pt);
\draw[fill] (7,5) circle (2.5pt);
\draw[fill] (7,7) circle (2.5pt);
\draw[fill] (8,4) circle (2.5pt);
\draw[fill] (8,6) circle (2.5pt);
\draw[fill] (9,5) circle (2.5pt);
\node at (6,1) {\footnotesize $\alg K_4^\emptyset= K_0(\alg {DP}_4)$};
%*********************
\draw (12,2) -- (13,3) -- (14,4) -- (15,5) -- (14,6) -- (13,5) -- (14,4);
\draw (14,6) -- (13,7) -- (12,8);
\draw  (12,2) -- (11,3) -- (12,4) -- (13,3) -- (12,4) --(13,5) -- (12,6) -- (11,7) -- (12,8) -- (13,7) -- (12,6);
\draw  (12,4) --(11,5) -- (12,6);
\draw[fill] (11,3) circle (2.5pt);
\draw[fill] (11,5) circle (2.5pt);
\draw[fill] (11,7) circle (2.5pt);
\draw[fill] (12,2) circle (2.5pt);
\draw[fill] (12,4) circle (2.5pt);
\draw[fill] (12,6) circle (2.5pt);
\draw[fill] (12,8) circle (2.5pt);
\draw[fill] (13,3) circle (2.5pt);
\draw[fill] (13,5) circle (2.5pt);
\draw[fill] (13,7) circle (2.5pt);
\draw[fill] (14,4) circle (2.5pt);
\draw[fill] (14,6) circle (2.5pt);
\draw[fill] (15,5) circle (2.5pt);
\node at (12,1) {\footnotesize $\alg K_4^{\uparrow(a_2,a_2)}$};
%*********************
\draw (18,2) -- (19,3) -- (20,4) -- (21,5) -- (20,6) -- (19,5) -- (20,4);
\draw (20,6) -- (19,7) -- (18,8);
\draw  (18,2) -- (17,3) -- (18,4) -- (19,3) -- (18,4) --(19,5) -- (18,6) -- (17,7) -- (18,8) -- (19,7) -- (18,6);
\draw  (18,4) --(17,5) -- (18,6);
\draw  (17,3) -- (16,4) -- (17,5) -- (16,6) -- (17,7);
\draw[fill] (16,4) circle (2.5pt);
\draw[fill] (16,6) circle (2.5pt);
\draw[fill] (17,3) circle (2.5pt);
\draw[fill] (17,5) circle (2.5pt);
\draw[fill] (17,7) circle (2.5pt);
\draw[fill] (18,2) circle (2.5pt);
\draw[fill] (18,4) circle (2.5pt);
\draw[fill] (18,6) circle (2.5pt);
\draw[fill] (18,8) circle (2.5pt);
\draw[fill] (19,3) circle (2.5pt);
\draw[fill] (19,5) circle (2.5pt);
\draw[fill] (19,7) circle (2.5pt);
\draw[fill] (20,4) circle (2.5pt);
\draw[fill] (20,6) circle (2.5pt);
\draw[fill] (21,5) circle (2.5pt);
\node at (18,1) {\footnotesize $\alg K_4^{\uparrow(a_3,a_2)}$};
\end{tikzpicture}
\end{center}
\caption{The proper admissible subalgebras of $K(\alg{DP}_4)$\label{d4}}
\end{figure}

\begin{figure}[htbp]
\begin{center}
\begin{tikzpicture}[scale=.4]
\draw (6.,-6.)-- (7.,-5.);
\draw (7.,-5.)-- (8.,-4.);
\draw (8.,-4.)-- (9.,-3.);
\draw (9.,-3.)-- (10.,-2.);
\draw (10.,-2.)-- (9.,-1.);
\draw (9.,-1.)-- (8.,0.);
\draw (8.,0.)-- (7.,1.);
\draw (7.,1.)-- (6.,2.);
\draw (6.,-6.)-- (5.,-5.);
\draw (5.,-5.)-- (6.,-4.);
\draw (6.,-4.)-- (7.,-3.);
\draw (7.,-3.)-- (8.,-2.);
\draw (8.,-2.)-- (9.,-1.);
\draw (8.,-2.)-- (7.,-1.);
\draw (7.,-1.)-- (6.,0.);
\draw (6.,0.)-- (5.,1.);
\draw (5.,1.)-- (6.,2.);
\draw (6.,0.)-- (7.,1.);
\draw (7.,-1.)-- (8.,0.);
\draw (8.,-2.)-- (9.,-3.);
\draw (7.,-3.)-- (8.,-4.);
\draw (6.,-4.)-- (7.,-5.);
\draw (16.,2.)-- (20.,-2.);
\draw (20.,-2.)-- (16.,-6.);
\draw (26.,2.)-- (30.,-2.);
\draw (30.,-2.)-- (26.,-6.);
\draw (6.,-16.)-- (10.,-12.);
\draw (10.,-12.)-- (6.,-8.);
\draw (16.,-8.)-- (20.,-12.);
\draw (16.,-16.)-- (20.,-12.);
\draw (26.,-8.)-- (30.,-12.);
\draw (30.,-12.)-- (26.,-16.);
\draw (24.,-14.)-- (28.,-10.);
\draw (25.,-15.)-- (29.,-11.);
\draw (24.,-10.)-- (28.,-14.);
\draw (29.,-13.)-- (25.,-9.);
\draw (24.,-10.)-- (26.,-8.);
\draw (24.,-12.)-- (27.,-9.);
\draw (24.,-12.)-- (27.,-15.);
\draw (24.,-14.)-- (26.,-16.);
\draw (16.,-16.)-- (15.,-15.);
\draw (15.,-15.)-- (19.,-11.);
\draw (16.,-8.)-- (15.,-9.);
\draw (15.,-9.)-- (19.,-13.);
\draw (16.,-10.)-- (17.,-9.);
\draw (17.,-11.)-- (18.,-10.);
\draw (17.,-13.)-- (18.,-14.);
\draw (17.,-15.)-- (16.,-14.);
\draw (18.,-2.)-- (19.,-3.);
\draw (18.,-2.)-- (19.,-1.);
\draw (16.,-2.)-- (18.,-4.);
\draw (16.,-2.)-- (18.,0.);
\draw (18.,-2.)-- (15.,-5.);
\draw (16.,-6.)-- (15.,-5.);
\draw (17.,-5.)-- (16.,-4.);
\draw (18.,-2.)-- (15.,1.);
\draw (15.,1.)-- (16.,2.);
\draw (16.,0.)-- (17.,1.);
\draw (26.,-6.)-- (25.,-5.);
\draw (27.,-5.)-- (25.,-3.);
\draw (25.,-3.)-- (28.,0.);
\draw (28.,-4.)-- (26.,-2.);
\draw (29.,-3.)-- (27.,-1.);
\draw (29.,-1.)-- (26.,-4.);
\draw (26.,2.)-- (25.,1.);
\draw (25.,1.)-- (27.,-1.);
\draw (26.,0.)-- (27.,1.);
\draw (25.,-5.)-- (26.,-4.);
\draw (26.,-2.)-- (25.,-1.);
\draw (25.,-1.)-- (26.,0.);
\draw (6.,-16.)-- (5.,-15.);
\draw (5.,-15.)-- (9.,-11.);
\draw (6.,-8.)-- (5.,-9.);
\draw (5.,-9.)-- (9.,-13.);
\draw (7.,-15.)-- (6.,-14.);
\draw (8.,-14.)-- (7.,-13.);
\draw (7.,-11.)-- (8.,-10.);
\draw (6.,-10.)-- (7.,-9.);
\draw (5.,-9.)-- (4.,-10.);
\draw (4.,-10.)-- (7.,-13.);
\draw (7.,-11.)-- (4.,-14.);
\draw (4.,-14.)-- (5.,-15.);
\draw (6.,-10.)-- (5.,-11.);
\draw (6.,-14.)-- (5.,-13.);
\draw (16.,-10.)-- (14.,-12.);
\draw (14.,-12.)-- (16.,-14.);
\draw (17.,-11.)-- (16.,-12.);
\draw (16.,-12.)-- (15.,-11.);
\draw (16.,-12.)-- (17.,-13.);
\draw (16.,-12.)-- (15.,-13.);
\draw [fill] (6.,-6.) circle (2.5pt);
\node[below] at (6.,-6.5) {\footnotesize $\alg K_5^\emptyset = K_0(\alg {DP}_5)$};
\draw [fill] (7.,-5.) circle (2.5pt);
\draw [fill] (8.,-4.) circle (2.5pt);
\draw [fill] (9.,-3.) circle (2.5pt);
\draw [fill] (10.,-2.) circle (2.5pt);
\draw [fill] (9.,-1.) circle (2.5pt);
\draw [fill] (8.,0.) circle (2.5pt);
\draw [fill] (7.,1.) circle (2.5pt);
\draw [fill] (6.,2.) circle (2.5pt);
\draw [fill] (5.,-5.) circle (2.5pt);
\draw [fill] (6.,-4.) circle (2.5pt);
\draw [fill] (7.,-3.) circle (2.5pt);
\draw [fill] (8.,-2.) circle (2.5pt);
\draw [fill] (7.,-1.) circle (2.5pt);
\draw [fill] (6.,0.) circle (2.5pt);
\draw [fill] (5.,1.) circle (2.5pt);
\draw [fill] (16.,2.) circle (2.5pt);
\draw [fill] (20.,-2.) circle (2.5pt);
\draw [fill] (16.,-6.) circle (2.5pt);
\node[below] at (16.,-6.5) {\footnotesize $\alg K_5^{\uparrow(a_2,a_2)}$};
\draw [fill] (26.,2.) circle (2.5pt);
\draw [fill] (30.,-2.) circle (2.5pt);
\draw [fill] (26.,-6.) circle (2.5pt);
\node[below] at (26.,-6.5) {\footnotesize $\alg K_5^{\uparrow(a_3,a_2)}$};
\draw [fill] (6.,-8.) circle (2.5pt);
\draw [fill] (6.,-16.) circle (2.5pt);
\node[below] at (6.,-16.5) {\footnotesize $\alg K_5^{\uparrow(a_4,a_2)}$};
\draw [fill] (10.,-12.) circle (2.5pt);
\draw [fill] (16.,-8.) circle (2.5pt);
\draw [fill] (20.,-12.) circle (2.5pt);
\draw [fill] (16.,-16.) circle (2.5pt);
\node[below] at (16.,-16.5) {\footnotesize $\alg K_5^{\uparrow(a_3,a_3)}$};
\draw [fill] (26.,-8.) circle (2.5pt);
\draw [fill] (30.,-12.) circle (2.5pt);
\draw [fill] (26.,-16.) circle (2.5pt);
\node[below] at (26.,-16.5) {\footnotesize $\alg K_5^{\uparrow(a_4,a_2),(a_3,a_3)}$};
\draw [fill] (25.,-15.) circle (2.5pt);
\draw [fill] (26.,-14.) circle (2.5pt);
\draw [fill] (27.,-13.) circle (2.5pt);
\draw [fill] (28.,-12.) circle (2.5pt);
\draw [fill] (29.,-13.) circle (2.5pt);
\draw [fill] (28.,-14.) circle (2.5pt);
\draw [fill] (27.,-15.) circle (2.5pt);
\draw [fill] (24.,-14.) circle (2.5pt);
\draw [fill] (24.,-12.) circle (2.5pt);
\draw [fill] (25.,-13.) circle (2.5pt);
\draw [fill] (25.,-11.) circle (2.5pt);
\draw [fill] (26.,-12.) circle (2.5pt);
\draw [fill] (26.,-10.) circle (2.5pt);
\draw [fill] (27.,-11.) circle (2.5pt);
\draw [fill] (25.,-9.) circle (2.5pt);
\draw [fill] (24.,-10.) circle (2.5pt);
\draw [fill] (27.,-9.) circle (2.5pt);
\draw [fill] (28.,-10.) circle (2.5pt);
\draw [fill] (29.,-11.) circle (2.5pt);
\draw [fill] (15.,-15.) circle (2.5pt);
\draw [fill] (19.,-11.) circle (2.5pt);
\draw [fill] (15.,-9.) circle (2.5pt);
\draw [fill] (19.,-13.) circle (2.5pt);
\draw [fill] (17.,-9.) circle (2.5pt);
\draw [fill] (18.,-10.) circle (2.5pt);
\draw [fill] (17.,-15.) circle (2.5pt);
\draw [fill] (18.,-14.) circle (2.5pt);
\draw [fill] (16.,-14.) circle (2.5pt);
\draw [fill] (17.,-13.) circle (2.5pt);
\draw [fill] (16.,-10.) circle (2.5pt);
\draw [fill] (17.,-11.) circle (2.5pt);
\draw [fill] (18.,-12.) circle (2.0pt);
\draw [fill] (18.,-2.) circle (2.5pt);
\draw [fill] (19.,-3.) circle (2.5pt);
\draw [fill] (19.,-1.) circle (2.5pt);
\draw [fill] (16.,-2.) circle (2.5pt);
\draw [fill] (18.,-4.) circle (2.5pt);
\draw [fill] (18.,0.) circle (2.5pt);
\draw [fill] (15.,-5.) circle (2.5pt);
\draw [fill] (17.,-5.) circle (2.5pt);
\draw [fill] (16.,-4.) circle (2.5pt);
\draw [fill] (15.,1.) circle (2.5pt);
\draw [fill] (16.,0.) circle (2.5pt);
\draw [fill] (17.,1.) circle (2.5pt);
\draw [fill] (25.,-5.) circle (2.5pt);
\draw [fill] (27.,-5.) circle (2.5pt);
\draw [fill] (25.,-3.) circle (2.5pt);
\draw [fill] (28.,0.) circle (2.5pt);
\draw [fill] (28.,-4.) circle (2.5pt);
\draw [fill] (26.,-2.) circle (2.5pt);
\draw [fill] (29.,-3.) circle (2.5pt);
\draw [fill] (27.,-1.) circle (2.5pt);
\draw [fill] (29.,-1.) circle (2.5pt);
\draw [fill] (26.,-4.) circle (2.5pt);
\draw [fill] (25.,1.) circle (2.5pt);
\draw [fill] (26.,0.) circle (2.5pt);
\draw [fill] (27.,1.) circle (2.5pt);
\draw [fill] (27.,-3.) circle (2.0pt);
\draw [fill] (28.,-2.) circle (2.5pt);
\draw [fill] (17.,-3.) circle (2.0pt);
\draw [fill] (17.,-1.) circle (2.5pt);
\draw [fill] (25.,-1.) circle (2.5pt);
\draw [fill] (5.,-15.) circle (2.5pt);
\draw [fill] (9.,-11.) circle (2.5pt);
\draw [fill] (5.,-9.) circle (2.5pt);
\draw [fill] (9.,-13.) circle (2.5pt);
\draw [fill] (7.,-15.) circle (2.5pt);
\draw [fill] (6.,-14.) circle (2.5pt);
\draw [fill] (8.,-14.) circle (2.5pt);
\draw [fill] (7.,-13.) circle (2.5pt);
\draw [fill] (7.,-11.) circle (2.5pt);
\draw [fill] (8.,-10.) circle (2.5pt);
\draw [fill] (6.,-10.) circle (2.5pt);
\draw [fill] (7.,-9.) circle (2.5pt);
\draw [fill] (4.,-10.) circle (2.5pt);
\draw [fill] (4.,-14.) circle (2.5pt);
\draw [fill] (5.,-11.) circle (2.5pt);
\draw [fill] (5.,-13.) circle (2.5pt);
\draw [fill] (14.,-12.) circle (2.5pt);
\draw [fill] (16.,-12.) circle (2.5pt);
\draw [fill] (15.,-11.) circle (2.5pt);
\draw [fill] (15.,-13.) circle (2.5pt);
\draw [fill] (6.,-12.) circle (2.0pt);
\draw [fill] (8.,-12.) circle (2.0pt);
\draw (16.,-18.)-- (20.,-22.);
\draw (16.,-18.)-- (13.,-21.);
\draw (13.,-23.)-- (17.,-19.);
\draw (14.,-24.)-- (18.,-20.);
\draw (13.,-23.)-- (16.,-26.);
\draw (16.,-26.)-- (20.,-22.);
\draw (19.,-21.)-- (15.,-25.);
\draw (13.,-21.)-- (17.,-25.);
\draw (18.,-24.)-- (14.,-20.);
\draw (15.,-19.)-- (19.,-23.);
\draw [fill] (16.,-26.) circle (2.5pt);
\node[below] at (16.,-26.) {\footnotesize $\alg K_5^{\uparrow(a_4,a_3)}$};
\draw [fill] (18.,-24.) circle (2.5pt);
\draw [fill] (20.,-22.) circle (2.5pt);
\draw [fill] (17.,-19.) circle (2.5pt);
\draw [fill] (15.,-19.) circle (2.5pt);
\draw [fill] (19.,-21.) circle (2.5pt);
\draw [fill] (19.,-23.) circle (2.5pt);
\draw [fill] (18.,-22.) circle (2.5pt);
\draw [fill] (17.,-21.) circle (2.5pt);
\draw [fill] (16.,-20.) circle (2.5pt);
\draw [fill] (17.,-23.) circle (2.5pt);
\draw [fill] (16.,-24.) circle (2.5pt);
\draw [fill] (15.,-25.) circle (2.5pt);
\draw [fill] (17.,-25.) circle (2.5pt);
\draw [fill] (15.,-23.) circle (2.5pt);
\draw [fill] (16.,-22.) circle (2.5pt);
\draw [fill] (15.,-21.) circle (2.5pt);
\draw [fill] (14.,-22.) circle (2.5pt);
\draw [fill] (13.,-21.) circle (2.5pt);
\draw [fill] (13.,-23.) circle (2.5pt);
\end{tikzpicture}
\end{center}
\caption{The proper admissible subalgebras of $K(\alg{DP}_5)$\label{d5}}
\end{figure}
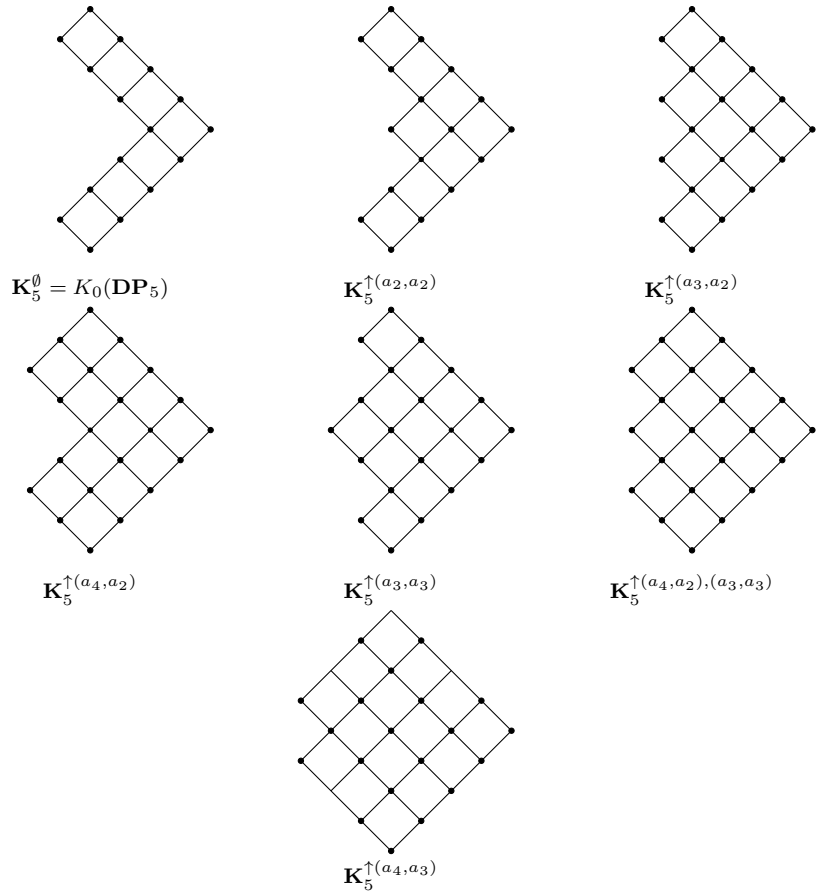

Since the non admissible subalgebras of $K(\alg {DP}_n)$ are exactly the admissible subalgebras of $K(\alg {DP}_m)$ with $m \le n$, in principle we can draw the entire lattice of subalgebras. For that we use the following result, in which we use that in the inclusion $\alg {DP}_m\leq \alg {DP}_n$, the coatom $a_{1}^{(m)}$ goes to the coatom $a_{1}^{(n)}$, each $a_{k}^{(m)}$ for $k=2,\ldots,m-2$ goes to $a_{k}^{(n)}$, and $a_{m-1}^{(m)}=0$ goes to $a_{n-1}^{(n)}=0$.

\begin{theorem}
\begin{enumerate}
	\item for each $4\leq n$, $\alg K_{0,2}\leq \alg K_{1,2}\leq \alg K_n^{\emptyset}$;
	\item for each $4\leq m<n$, $\alg K_m^U\leq \alg K_n^{\tilde{U}}$, where $\tilde{U}$ is the up-set of $X_n$ generated by each $(a_i,a_j)\in U$ with $i,j<m-1$, each $(a_{n-1},a_j)$ if $(a_{m-1},a_j)\in U$ and $j<m-1$, and $(a_{n-1},a_{n-1})$ if $(a_{m-1},a_{m-1})\in U$.
\end{enumerate}
\end{theorem}

From here, and using the fact that none of $\alg K_n^U$ has proper non-trivial congruences, using the same techniques that we have employed throughout this paper, we can describe $\Lambda(K(\alg{DP}_n))$.

The lattices $\Lambda(K(\alg{DP}_4))$ and $\Lambda(K(\alg{DP}_5))$ are in Figures \ref{dp4} and \ref{dp5}, respectively; we have used the same names before for the admissible subalgebras of $K(\alg{DP}_3) = K(\alg \L_2)$. Unfortunately $\Lambda(K(\alg{DP}_n))$ becomes very intricate as soon as $n>4$, and it is borderline impossible to draw it by hands.

\begin{figure}[htbp]
\begin{center}
\begin{tikzpicture}[scale=.5]
\draw (0,0) -- (0,1) -- (1,1.5) -- (2,2) -- (3,2.5) -- (4,3) -- (5,3.5) -- (5,4.5);
\draw (0,1) -- (0,2) -- (1,2.5) -- (2,3) -- (3,3.5) -- (4,4) -- (5,4.5) -- (5,5.5);
\draw (2,3) -- (2,4) -- (3,4.5) -- (4,5) -- (5,5.5) -- (5,6.5);
\draw (1,1.5) -- (1,2.5);
\draw (2,2) -- (2,3);
\draw (3,2.5) -- (3,3.5) -- (3,4.5);
\draw (4,3) -- (4,4) -- (4,5);
\draw[fill] (0,0) circle [radius=0.05]; \node[below] at (0,0) {\tiny $\mathsf{T}$};
\draw[fill] (0,1) circle [radius=0.05]; \node[left] at (0,1) {\tiny $\VV(\alg K_3$)};
\draw[fill] (0,2) circle [radius=0.05]; \node[left] at (0,2) {\tiny $\VV(\alg K_4)$};
\draw[fill] (1,1.5) circle [radius=0.05]; \node[right] at (1,1.4) {\tiny $\VV(\alg K_{0,2}$)};
\draw[fill] (1,2.5) circle [radius=0.05];
\draw[fill] (2,2) circle [radius=0.05]; \node[right] at (2,1.9) {\tiny $\VV(\alg K_{1,2}$)};
\draw[fill] (2,3) circle [radius=0.05];
\draw[fill] (2,4) circle [radius=0.05]; \node[left] at (2,4) {\tiny $\VV(K(\alg \L_2) )= \VV(K_{2,2})$};
\draw[fill] (3,2.5) circle [radius=0.05]; \node[right] at (3,2.4) {\tiny $\VV(\alg K_4^\emptyset)$};
\draw[fill] (3,3.5) circle [radius=0.05];
\draw[fill] (3,4.5) circle [radius=0.05];
\draw[fill] (4,3) circle [radius=0.05]; \node[right] at (4,2.92) {\tiny $\VV(\alg K_4^{\uparrow(a_2,a_2)})$};
\draw[fill] (4,4) circle [radius=0.05];
\draw[fill] (4,5) circle [radius=0.05];
\draw[fill] (5,3.5) circle [radius=0.05]; \node[right] at (5,3.5)  {\tiny $\VV(\alg K_4^{\uparrow(a_3,a_2)})$};
\draw[fill] (5,4.5) circle [radius=0.05];
\draw[fill] (5,5.5) circle [radius=0.05];
\draw[fill] (5,6.5) circle [radius=0.05]; \node[right] at (5,6.5) {\tiny $\VV(K(\alg{DP}_4))= \VV(\alg K_4^{\uparrow(a_3,a_3)})$};
\end{tikzpicture}
\end{center}
\caption{The lattice  $\Lambda(K(\alg{DP}_4))$\label{dp4}}
\end{figure}
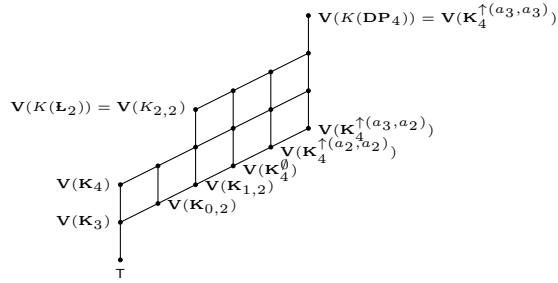

\begin{figure}[htbp]
\begin{center}
\begin{tikzpicture}[scale=.2]
\draw (18.,-22.)-- (18.,-20.);
\draw (18.,-20.)-- (22.,-18.);
\draw (18.,-20.)-- (18.,-18.);
\draw (22.,-18.)-- (22.,-16.);
\draw (18.,-18.)-- (22.,-16.);
\draw (22.,-16.)-- (26.,-14.);
\draw (26.,-14.)-- (26.,-16.);
\draw (26.,-16.)-- (22.,-18.);
\draw (26.,-14.)-- (26.,-12.);
\draw (26.,-16.)-- (24.,-14.);
\draw (24.,-14.)-- (24.,-12.);
\draw (24.,-12.)-- (24.,-10.);
\draw (26.,-14.)-- (24.,-12.);
\draw (26.,-12.)-- (24.,-10.);
\draw (24.,-14.)-- (22.,-12.);
\draw (22.,-12.)-- (20.,-10.);
\draw (24.,-12.)-- (22.,-10.);
\draw (22.,-12.)-- (22.,-10.);
\draw (22.,-10.)-- (20.,-8.);
\draw (20.,-10.)-- (20.,-8.);
\draw (24.,-10.)-- (22.,-8.);
\draw (22.,-10.)-- (22.,-8.);
\draw (22.,-8.)-- (20.,-6.);
\draw (20.,-8.)-- (20.,-6.);
\draw (20.,-6.)-- (20.,-4.);
\draw (28.,-10.)-- (28.,-12.);
\draw (28.,-10.)-- (28.,-8.);
\draw (24.,-14.)-- (28.,-12.);
\draw (24.,-12.)-- (28.,-10.);
\draw (24.,-10.)-- (28.,-8.);
\draw (22.,-12.)-- (26.,-10.);
\draw (28.,-12.)-- (26.,-10.);
\draw (26.,-10.)-- (26.,-8.);
\draw (26.,-8.)-- (26.,-6.);
\draw (26.,-6.)-- (22.,-8.);
\draw (26.,-8.)-- (22.,-10.);
\draw (26.,-8.)-- (28.,-10.);
\draw (26.,-6.)-- (28.,-8.);
\draw (26.,-10.)-- (30.,-8.);
\draw (30.,-6.)-- (26.,-8.);
\draw (30.,-6.)-- (30.,-4.);
\draw (30.,-4.)-- (26.,-6.);
\draw (30.,-8.)-- (30.,-6.);
\draw (20.,-10.)-- (24.,-8.);
\draw (24.,-8.)-- (26.,-10.);
\draw (20.,-8.)-- (24.,-6.);
\draw (24.,-6.)-- (26.,-8.);
\draw (24.,-8.)-- (24.,-6.);
\draw (24.,-6.)-- (24.,-4.);
\draw (24.,-4.)-- (20.,-6.);
\draw (24.,-4.)-- (26.,-6.);
\draw (24.,-2.)-- (24.,-4.);
\draw (24.,-2.)-- (20.,-4.);
\draw (30.,-8.)-- (28.,-6.);
\draw (28.,-6.)-- (24.,-8.);
\draw (24.,-6.)-- (28.,-4.);
\draw (28.,-6.)-- (28.,-4.);
\draw (24.,-4.)-- (28.,-2.);
\draw (28.,-4.)-- (28.,-2.);
\draw (28.,-2.)-- (28.,0.);
\draw (28.,0.)-- (24.,-2.);
\draw (30.,-6.)-- (28.,-4.);
\draw (30.,-4.)-- (28.,-2.);
\draw (30.,-8.)-- (34.,-6.);
\draw (28.,-6.)-- (32.,-4.);
\draw (32.,-4.)-- (34.,-6.);
\draw (34.,-6.)-- (34.,-4.);
\draw (34.,-4.)-- (34.,-2.);
\draw (32.,-4.)-- (32.,-2.);
\draw (32.,-2.)-- (32.,0.);
\draw (32.,0.)-- (28.,-2.);
\draw (32.,2.)-- (28.,0.);
\draw (32.,0.)-- (32.,2.);
\draw (32.,-2.)-- (28.,-4.);
\draw (30.,-4.)-- (34.,-2.);
\draw (30.,-6.)-- (34.,-4.);
\draw (32.,-2.)-- (34.,-4.);
\draw (32.,0.)-- (34.,-2.);
\draw (30.,-2.)-- (32.,-4.);
\draw (30.,-2.)-- (30.,0.);
\draw (30.,0.)-- (30.,2.);
\draw (30.,2.)-- (30.,4.);
\draw (30.,4.)-- (32.,2.);
\draw (32.,0.)-- (30.,2.);
\draw (32.,-2.)-- (30.,0.);
\draw (34.,-6.)-- (38.,-4.);
\draw (38.,-4.)-- (38.,-2.);
\draw (34.,-4.)-- (38.,-2.);
\draw (34.,-2.)-- (38.,0.);
\draw (38.,-2.)-- (38.,0.);
\draw (32.,-4.)-- (36.,-2.);
\draw (32.,-2.)-- (36.,0.);
\draw (32.,0.)-- (36.,2.);
\draw (30.,-2.)-- (34.,0.);
\draw (30.,0.)-- (34.,2.);
\draw (30.,2.)-- (34.,4.);
\draw (30.,4.)-- (34.,6.);
\draw (32.,2.)-- (36.,4.);
\draw (36.,-2.)-- (38.,-4.);
\draw (36.,0.)-- (36.,-2.);
\draw (36.,0.)-- (38.,-2.);
\draw (36.,0.)-- (36.,2.);
\draw (36.,2.)-- (38.,0.);
\draw (36.,-2.)-- (34.,0.);
\draw (34.,0.)-- (34.,2.);
\draw (34.,2.)-- (36.,0.);
\draw (36.,2.)-- (34.,4.);
\draw (34.,4.)-- (34.,2.);
\draw (34.,4.)-- (34.,6.);
\draw (36.,2.)-- (36.,4.);
\draw (36.,4.)-- (34.,6.);
\draw (34.,0.)-- (38.,2.);
\draw (34.,2.)-- (38.,4.);
\draw (34.,4.)-- (38.,6.);
\draw (34.,6.)-- (38.,8.);
\draw (38.,2.)-- (38.,4.);
\draw (38.,4.)-- (38.,6.);
\draw (38.,6.)-- (38.,8.);
\draw (38.,2.)-- (42.,4.);
\draw (38.,4.)-- (42.,6.);
\draw (38.,6.)-- (42.,8.);
\draw (38.,8.)-- (42.,10.);
\draw (42.,10.)-- (42.,8.);
\draw (42.,8.)-- (42.,6.);
\draw (42.,6.)-- (42.,4.);
\draw (42.,12.)-- (42.,10.);
\draw [fill=black] (18.,-22.) circle (5pt);
\node[below] at (18.,-22.) {\tiny $\mathsf{T}$};
\draw [fill=gray] (18.,-20.) circle (5pt);
\node[left] at (18.,-20.) {\tiny $\VV(\alg K_3)$};
\draw [fill=gray] (22.,-18.) circle (5pt);
\node[right] at (22.,-18.) {\tiny $\VV(K_{0}(\alg(L_{2}))) = \VV(\alg K_{0,2})$};
\draw [fill=gray] (18.,-18.) circle (5pt);
\node[left] at (18.,-18.) {\tiny $K(\mathsf{BA}=\VV(\alg K_4)$};
\draw [fill=black] (22.,-16.) circle (5pt);
\draw [fill=gray] (26.,-16.) circle (5pt);
\node[right] at (26.,-16.) {\tiny $\VV(\alg K_{1,2})$};
\draw [fill=black] (26.,-14.) circle (5pt);
\draw [fill=gray] (26.,-12.) circle (5pt);
\node[below right] at (26.,-12.) {\tiny $\VV(K(\alg(L_{2})))$};
\draw [fill=gray] (24.,-14.) circle (5pt);
\node[left] at (24.,-14.) {\tiny $K_{0}(\alg{DP}_{4})=\alg K_4^\emptyset$};
\draw [fill=black] (24.,-12.) circle (5pt);
\draw [fill=black] (24.,-10.) circle (5pt);
\draw [fill=gray] (22.,-12.) circle (5pt);
\node[left] at (22.,-12.) {\tiny $\VV(\alg K_4^{\uparrow(a_2,a_2)})$};
\draw [fill=gray] (20.,-10.) circle (5pt);
\node[left] at (20.,-10.) {\tiny $\VV(\alg K_4^{\uparrow(a_3,a_2)})$};
\draw [fill=black] (20.,-8.) circle (5pt);
\draw [fill=black] (20.,-6.) circle (5pt);
\draw [fill=gray] (20.,-4.) circle (5pt);
\node[above left] at (20.,-4.) {\tiny $\VV(K(\alg{DP}_{4}))$};
\draw [fill=black] (22.,-10.) circle (5pt);
\draw [fill=black] (22.,-8.) circle (5pt);
\draw [fill=gray] (28.,-12.) circle (5pt);
\node[above right] at (28.,-12.) {\tiny $K_{0}(\alg{DP}_{5})=\alg K_5^\emptyset$};
\draw [fill=black] (28.,-10.) circle (5pt);
\draw [fill=black] (28.,-8.) circle (5pt);
\draw [fill=black] (26.,-10.) circle (5pt);
\draw [fill=black] (26.,-8.) circle (5pt);
\draw [fill=black] (26.,-6.) circle (5pt);
\draw [fill=gray] (30.,-8.) circle (5pt);
\node[right] at (30.,-8.) {\tiny $\VV(\alg K_5^{\uparrow(a_2,a_2)})$};
\draw [fill=black] (30.,-6.) circle (5pt);
\draw [fill=black] (30.,-4.) circle (5pt);
\draw [fill=black] (24.,-8.) circle (5pt);
\draw [fill=black] (24.,-6.) circle (5pt);
\draw [fill=black] (24.,-4.) circle (5pt);
\draw [fill=black] (24.,-2.) circle (5pt);
\draw [fill=black] (28.,-6.) circle (5pt);
\draw [fill=black] (28.,-2.) circle (5pt);
\draw [fill=black] (28.,0.) circle (5pt);
\draw [fill=black] (28.,-4.) circle (5pt);
\draw [fill=black] (34.,-4.) circle (5pt);
\draw [fill=gray] (34.,-6.) circle (5pt);
\node[right] at (34.,-6.) {\tiny $\VV(\alg K_5^{\uparrow(a_3,a_2)})$};
\draw [fill=black] (32.,-4.) circle (5pt);
\draw [fill=black] (34.,-2.) circle (5pt);
\draw [fill=black] (32.,0.) circle (5pt);
\draw [fill=black] (32.,2.) circle (5pt);
\draw [fill=black] (32.,-2.) circle (5pt);
\draw [fill=gray] (30.,-2.) circle (5pt);
\node[above left] at (30.,-2.) {\tiny $\VV(\alg K_5^{\uparrow(a_4,a_2)})$};
\draw [fill=black] (30.,0.) circle (5pt);
\draw [fill=black] (30.,2.) circle (5pt);
\draw [fill=black] (30.,4.) circle (5pt);
\draw [fill=gray] (38.,-4.) circle (5pt);
\node[right] at (38.,-4.) {\tiny $\VV(\alg K_5^{\uparrow(a_3,a_3)})$};
\draw [fill=black] (38.,-2.) circle (5pt);
\draw [fill=black] (38.,0.) circle (5pt);
\draw [fill=black] (36.,-2.) circle (5pt);
\draw [fill=black] (36.,0.) circle (5pt);
\draw [fill=black] (36.,2.) circle (5pt);
\draw [fill=black] (34.,0.) circle (5pt);
\draw [fill=black] (34.,2.) circle (5pt);
\draw [fill=black] (34.,4.) circle (5pt);
\draw [fill=black] (34.,6.) circle (5pt);
\draw [fill=black] (36.,4.) circle (5pt);
\draw [fill=gray] (38.,2.) circle (5pt);
\node[right] at (38.,2.) {\tiny $\VV(\alg K_5^{\uparrow(a_4,a_2),(a_3,a_3)})$};
\draw [fill=black] (38.,4.) circle (5pt);
\draw [fill=black] (38.,6.) circle (5pt);
\draw [fill=black] (38.,8.) circle (5pt);
\draw [fill=gray] (42.,12.) circle (5pt);
\node[above] at (42.,12.) {\tiny $\VV(K(\alg{DP}_{5}))$};
\draw [fill=gray] (42.,4.) circle (5pt);
\node[right] at (42.,4.) {\tiny $\VV(\alg K_5^{\uparrow(a_4,a_3)})$};
\draw [fill=black] (42.,6.) circle (5pt);
\draw [fill=black] (42.,8.) circle (5pt);
\draw [fill=black] (42.,10.) circle (5pt);
\end{tikzpicture}
\end{center}
\caption{The lattice  $\Lambda(K(\alg{DP}_5))$\label{dp5}}
\end{figure}
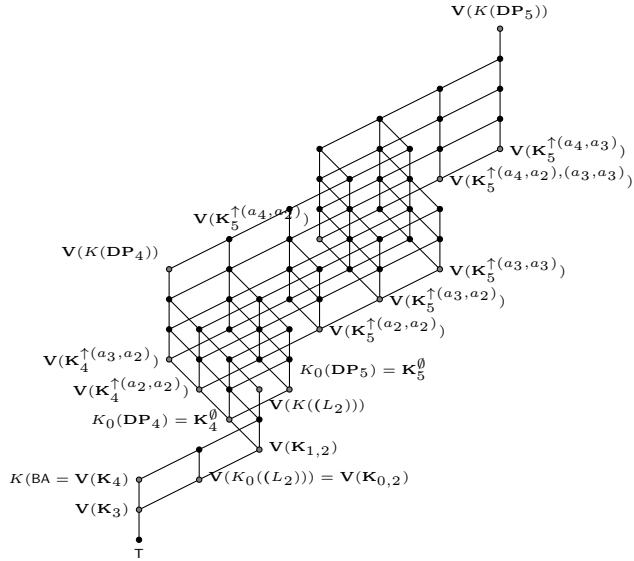

\subsection{Stonean K-lattices}\label{Stone}

\textbf{Stonean residuated lattices} are bounded residuated lattices satisfying
\begin{align*} \neg x\vee \neg\neg x \app 1.\end{align*}

We call the variety of Stonean residuated lattices $\mathsf{SRL}$.
Here we list some properties of Stonean residuated lattices that will be useful in what follows; they can be found in \cite{BusanicheCignoliMarcos2019}

\begin{itemize}
	\item Stonean residuated lattices are pseudocomplemented, i.e. they satisfy $x\wedge \neg x = 0$.
	\item They satisfy $\neg(x\wedge y)=\neg x\vee\neg y$. Moreover all elements of the form $\neg x$ are Boolean.
	\item They satisfy $x=\neg\neg x(\neg\neg x\to x)$, and as $\neg\neg x$ is Boolean this is equivalent to $x=\neg\neg x\wedge(\neg x\vee x)$.
	\item All directly indecomposable Stonean residuated lattices are of the form $\alg 2\oplus\alg D$ for some $\alg D\in\mathsf{CIRL}$ and viceversa.
	\item As subdirectly irreducible algebras are directly indecomposable, and filters in $\alg 2\oplus\alg D\in\mathsf{SRL}$	are either filters of $\alg D$ or the whole algebra $\alg 2\oplus\alg D$, subdirectly irreducible algebras in $\mathsf{SRL}$ are of the form $\alg 2\oplus\alg D$ for some subdirectly irreducible $\alg D\in\mathsf{CIRL}$.
\end{itemize}

In \cite{BusanicheCignoli2014}, as pseudocomplemented lattices satisfy the Glivenko equation
\begin{align*}\neg\neg (\neg\neg x\to x)=1,\end{align*}
they present in Theorem 5.17 a bijective corresponence between good lattice filters of a pseudocomplemented distributive residuated lattice $\alg A$ and good admissible subalgebras of $K(\alg A)$. This result can be adapted and improved for Stonean residuated lattices.

Let $\alg A\in\mathsf{SRL}$; a lattice filter $F$ in $\alg A$ is called \textsl{good} if $\neg \neg x\in F$ implies $x\in F$.

\begin{lemma}Good lattice filters in $\alg A\in\mathsf{SRL}$ are exactly the lattice filters containing all dense elements $D=\{x\in A:\neg x=0\}$.\end{lemma}
\begin{proof}If $F$ is a good filter and $\neg x = 0$, then clearly $\neg\neg x=1\in F$, so $x\in F$ and the filter contains all dense elements. Reciprocally, if $F$ is a lattice filter containing all dense elements and $\neg\neg x\in F$, recalling that $x=\neg\neg x\wedge(\neg x\vee x)$ and that $\neg x\vee x$ is dense, $\neg\neg x,\neg x\vee x \in F$ so $x\in F$.\end{proof}

We say that an admissible subalgebra $\alg S$ of $K(\alg A)$ is \textsl{good} provided $(\neg\neg a,\neg\neg b) \in S$ implies $(a,b) \in S$.

\begin{lemma}Let $\alg A\in\mathsf{SRL}$. All admissible subalgebras of $K(\alg A)$ are good.\end{lemma}
\begin{proof}Let $\alg S$ be an admissible subalgebra of $K(\alg A)$ and suppose $(\neg\neg a,\neg\neg b) \in S$. As $(1,b),(a,\neg a),(\neg b,b),(\neg\neg b,\neg\neg a)\in S$ and recalling that the elements $\neg a,\neg b$ are Boolean, the following element is in $S$:
\begin{align*}\left(\left((\neg b,b)\wedge(\neg\neg b,\neg\neg a)\right) \vee (a,\neg a)\right)\wedge(1,b) &= \left(\left((0,b\vee \neg\neg a)\right) \vee (a,\neg a)\right)\wedge(1,b)\\
&= \left(a,(b\vee \neg\neg a)\wedge \neg a \right)\wedge(1,b)\\
&= \left(a,(b\wedge \neg a)\vee (\neg\neg a\wedge \neg a)\right)\wedge(1,b)\\
&= (a,b\wedge \neg a)\wedge(1,b)\\
&= (a,b).\end{align*}
\end{proof}

From Theorem \ref{admissibleordinalsum} it is clear that the only admissible subalgebras of $K(\alg 2\oplus\alg D)$ for $\alg D\in\mathsf{CIRL}$ are $K(\alg 2\oplus\alg D)$ and the subalgebra with universe $K(2\oplus D)\setminus \{(0,0)\}$. However, from the previous results and Theorem 6.10 in \cite{BusanicheCignoli2014} we have a better description of admissible subalgebras for $K(\alg A)$ where $\alg A$ is any Stonean residuated lattice.

\begin{theorem}For each $\alg A\in\mathsf{SRL}$, the correspondence
\begin{align*}F \mapsto \alg S_F \end{align*}
defines a bijection from the set of lattice filters of $\alg A$ containing all dense elements onto the set of all admissible subalgebras of $K(\alg A)$, where
$\alg S_F$ is the subalgebra of $K(\alg A)$ with universe
\begin{align*}\{(a,b)\in K(A):\neg a\to\neg\neg b\in F\}.\end{align*}
\end{theorem}

Observe that in particular if $\alg A\in\mathsf{SRL}$ is directly indecomposable, then there is only one proper filter containing all dense elements, so the only proper admissible subalgebra of $K(\alg A)$ will have universe $K(A)\setminus \{(0,0)\}$.

\medskip

With these results in consideration, we can investigate the lattice $\Lambda(K(\mathsf{SRL}))$. Not surprisingly it has highly complex, but we will get some information about the bottom part of the lattice.

Clearly $\VV(\alg K_3)$ is the only atom, and the next result is also immediate.

\begin{theorem}The only finitely generated almost minimal varieties of $K(\mathsf{SRL})$ are $\VV(\alg K_4)=K(\mathsf{BA})$ and $\VV(\alg K_8)=\VV(K_0(\alg 2\oplus\alg 2))$.\end{theorem}
\begin{proof}It is clear that $\VV(\alg K_4)$ is an almost minimal variety. Moreover, the only (up to isomorphism) rigid finite Stonean residuated lattice is $\alg 2\oplus \alg 2=\alg G_3$, and the conclusion follows.\end{proof}

From the properties of Stonean residuated lattices, we can say something more going upwards in the lattice of subvarieties of $K(\mathsf{SRL})$.

\begin{lemma}If $\alg A\in\mathsf{SRL}$ and $f:\alg A\to\alg 2$ is a morphism, then $f(x)=0$ implies $x=0$.\end{lemma}
\begin{proof}It is enough to show this for subdirectly irreducible algebras. In this case we have that $\alg A\cong \alg 2\oplus\alg D$ for some $\alg R\in\mathsf{CIRL}$, and if $x\in D$ $\neg x=0$, so $f(\neg x)=0$ and $f(x)=1$.\end{proof}

\begin{corollary}If $\alg A\in\mathsf{SRL}$ is subdirectly irreducible, then $\alg K_4\not\in\VV(K_0(\alg A))$.\end{corollary}
\begin{proof}Recall that ``$\alg K_4 \not \le \alg A$'' is first-order definable, so as it holds in $K_0(\alg A)$ it will be true for any ultrapower. From the previous result we have that $\alg K_4\not\in \HH\SU\PP_u(K_0(\alg A))$.\end{proof}

\begin{lemma}Let $\alg D\in\mathsf{CIRL}$ be finite and subdirectly irreducible. Then $\VV(\alg 2\oplus\alg D)$ covers $\VV(\alg 2\oplus\alg 2)$ in $\Lambda(\mathsf{SRL})$ if and only if $\VV(\alg D)$ covers $\VV(\alg 2)=\mathsf{GBA}$ in $\Lambda(\mathsf{CIRL})$.\end{lemma}

From these results, we obtain the following.

\begin{theorem}Let $\alg A\in K(\mathsf{SRL})$ be finite and subdirectly irreducible. Then $\VV(\alg A)$ is a cover of $\VV(\alg K_8)$ but not of $\VV(\alg K_4)=K(\mathsf{BA})$ if and only if $\alg A\cong K_0(\alg 2\oplus \alg D)$ for some $\alg D\in\mathsf{CIRL}$ finite and subdirectly irreducible that generates a cover of $\VV(\alg 2)=\mathsf{GBA}$ in $\Lambda(\mathsf{CIRL})$.\end{theorem}

Therefore if $\alg D\in\mathsf{CIRL}$ is finite, subdirectly irreducible and generates a cover of $\VV(\alg 2)=\mathsf{GBA}$ in $\Lambda(\mathsf{CIRL})$, the lattice of subvarieties of $\VV(K(\alg 2\oplus \alg D))$ will be as in Figure \ref{stlattice}.

\begin{figure}[htbp]
\begin{center}
\begin{tikzpicture}
\draw[fill] (0,0) circle [radius=0.05];
\draw[fill] (0,1) circle [radius=0.05];
\draw[fill] (1.5,1.5) circle [radius=0.05];
\draw[fill] (3,2) circle [radius=0.05];
%\draw[fill] (-1,1.5) circle [radius=0.05];
%\draw[fill] (0.5,2) circle [radius=0.05];
%\draw[fill] (2,2.5) circle [radius=0.05];
\draw[fill] (0,1+1) circle [radius=0.05];
\draw[fill] (1.5,1.5+1) circle [radius=0.05];
\draw[fill] (3,2+1) circle [radius=0.05];
%\draw[fill] (-1,1.5+1) circle [radius=0.05];
%\draw[fill] (0.5,2+1) circle [radius=0.05];
%\draw[fill] (2,2.5+1) circle [radius=0.05];
%\draw[fill] (0,1+2) circle [radius=0.05];
\draw[fill] (1.5,1.5+2) circle [radius=0.05];
\draw[fill] (3,2+2) circle [radius=0.05];
%\draw[fill] (-1,1.5+2) circle [radius=0.05];
%\draw[fill] (0.5,2+2) circle [radius=0.05];
%\draw[fill] (2,2.5+2) circle [radius=0.05];
%\draw[fill] (2,5.5) circle [radius=0.05];
\draw[fill] (3,5) circle [radius=0.05];
\draw (0,0) -- (0,1) -- (0,2);
\draw (0,1) -- (1.5,1.5) -- (3,2);
\draw (0,1+1) -- (1.5,1.5+1) -- (3,2+1);
\draw (1.5,1.5+2) -- (3,2+2);
\draw (0+1.5,0+1.5) -- (0+1.5,1+1.5) -- (0+1.5,2+1.5);
\draw (0+3,0+2) -- (0+3,1+2) -- (0+3,2+2);
\draw (3,4) -- (3,5);
\node[right] at (0,0) {\tiny $\mathsf{T}$};
\node[left] at (0,1) {\tiny $\mathsf{PKL} =\VV(\alg K_3)$};
\node[left] at (0,2) {\tiny $K(\mathsf{BA}) = \VV(\alg K_4)$};
\node[right] at (1.5,1.5) {\tiny $\VV(\alg K_8)$};
\node[right] at (3,2) {\tiny $\VV(K_0(\alg 2\oplus \alg D))$};
\node[left] at (1.5,3.5) {\tiny $\VV(\alg K_9)$};
\node[above] at (3,5) {\tiny $\VV(K(\alg 2\oplus \alg D))$};
\end{tikzpicture}
\end{center}
\caption{$\Lambda(\VV(K(\alg 2\oplus \alg D)))$\label{stlattice}}
\end{figure}
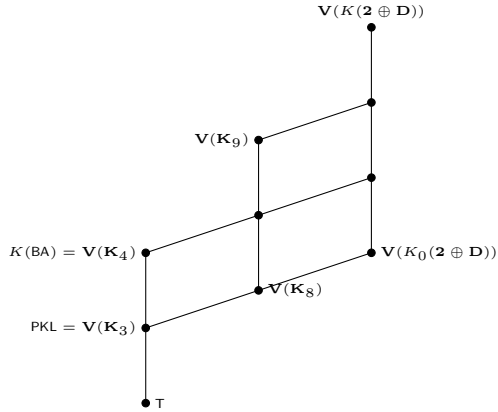

\medskip

For the case of non-finitely generated almost minimal varieties, an example will be $\VV(K_0(\alg 2\oplus \alg C_\o))$. If $\alg A\in\mathsf{CIRL}$ is infinite and generates an almost minimal variety in $\Lambda(\mathsf{CIRL})$ different from cancellative hoops $\mathsf{CH}$, then $\VV(K_0(\alg 2\oplus \alg A))$ will be another example of the cover.
In Figure \ref{stlattice2} we consider the lattice of subvarieties of $\VV(K(\alg 2\oplus \alg C_\o),K(\alg 2\oplus \alg \L_2))$

\begin{figure}[htbp]
\begin{center}
\begin{tikzpicture}
\draw[fill] (0,0) circle [radius=0.05];
\draw[fill] (0,1) circle [radius=0.05];
\draw[fill] (1.5,1.5) circle [radius=0.05];
\draw[fill] (3,2) circle [radius=0.05];
\draw[fill] (-1,1.5) circle [radius=0.05];
\draw[fill] (0.5,2) circle [radius=0.05];
\draw[fill] (2,2.5) circle [radius=0.05];
\draw[fill] (0,1+1) circle [radius=0.05];
\draw[fill] (1.5,1.5+1) circle [radius=0.05];
\draw[fill] (3,2+1) circle [radius=0.05];
\draw[fill] (-1,1.5+1) circle [radius=0.05];
\draw[fill] (0.5,2+1) circle [radius=0.05];
\draw[fill] (2,2.5+1) circle [radius=0.05];
%\draw[fill] (0,1+2) circle [radius=0.05];
\draw[fill] (1.5,1.5+2) circle [radius=0.05];
\draw[fill] (3,2+2) circle [radius=0.05];
\draw[fill] (-1,1.5+2) circle [radius=0.05];
\draw[fill] (0.5,2+2) circle [radius=0.05];
\draw[fill] (2,2.5+2) circle [radius=0.05];
\draw[fill] (2,5.5) circle [radius=0.05];
\draw[fill] (3,5) circle [radius=0.05];
\draw (0,0) -- (0,1) -- (0,2);
\draw (0,1) -- (1.5,1.5) -- (3,2) -- (2,2.5) -- (0.5,2) -- (-1,1.5) -- (0,1);
\draw (0,1+1) -- (1.5,1.5+1) -- (3,2+1) -- (2,2.5+1) -- (0.5,2+1) -- (-1,1.5+1) -- (0,1+1);
\draw (1.5,1.5+2) -- (3,2+2) -- (2,2.5+2) -- (0.5,2+2) -- (-1,1.5+2);
\draw (0-1,0+1.5) -- (0-1,1+1.5) -- (0-1,2+1.5);
\draw (0+1.5,0+1.5) -- (0+1.5,1+1.5) -- (0+1.5,2+1.5);
\draw (0+3,0+2) -- (0+3,1+2) -- (0+3,2+2);
\draw (0+2,0+2.5) -- (0+2,1+2.5) -- (0+2,2+2.5);
\draw (0+.5,0+2) -- (0+.5,1+2) -- (0+.5,2+2);
\draw (0+.5,0+2) -- (0+1.5,1+0.5);
\draw (0+.5,0+2+1) -- (0+1.5,1+0.5+1);
\draw (0+.5,0+2+2) -- (0+1.5,1+0.5+2);
\draw (3,4) -- (3,5) -- (2,5.5) -- (2,4.5);
\node[right] at (0,0) {\tiny $\mathsf{T}$};
\node[left] at (0,1) {\tiny $\mathsf{PKL} =\VV(\alg K_3)$};
\node[left] at (0,2) {\tiny $K(\mathsf{BA}) = \VV(\alg K_4)$};
\node[right] at (1.5,1.5) {\tiny $\VV(\alg K_8)$};
\node[right] at (3,2) {\tiny $\VV(K_0(\alg 2\oplus \alg \L_2))$};
\node[left] at (1.5,3.5) {\tiny $\VV(\alg K_9)$};
\node[right] at (3,5) {\tiny $\VV(K(\alg 2\oplus \alg \L_2))$};
\node[left] at (-1,1.5) {\tiny $\VV(K_0(\alg 2\oplus \alg C_\o))$};
\node[left] at (-1,3.5) {\tiny $K(\mathsf{PA})=\VV(K(\alg 2\oplus \alg C_\o))$};
\node[above] at (2,5.5) {\tiny $\VV(K(\alg 2\oplus \alg C_\o),K(\alg 2\oplus \alg \L_2))$};
\end{tikzpicture}
\end{center}
\caption{$\Lambda(\VV(K(\alg 2\oplus \alg C_\o),K(\alg 2\oplus \alg \L_2)))$\label{stlattice2}}
\end{figure}

\section*{Conclusions and future work}

 One of the most interesting features of the Kalman construction (at least from our point of view) is that it allows us to explore previously unexplored parts of the lattice of subvarieties of $\mathsf{CRL}$, by {\em lifting} properties of {\em integral} commutative residuated lattices that we already know.  In fact Section \ref{S3.examples} of this paper and Section 6 of \cite{AglianoMarcos2020a} are completely devoted to the purpose and we believe that we have shown that such an enterprize is at least useful. In particular the combination of the Kalman construction with other well known construction, such as the ordinal sum or the disconnected rotation, seems to be a very powerful tool and deserves to be investigated more.

Another possibility is to concentrate on a specific variety, e.g.  $K(\mathsf{BL})$ and investigate its algebraic properties. What are its splitting algebras? What are its projective members? Is there a canonical representation for subdirectly irreducible algebras? We believe that these are all very interesting questions and we propose to investigate them further.

\providecommand{\bysame}{\leavevmode\hbox to3em{\hrulefill}\thinspace}
\providecommand{\MR}{\relax\ifhmode\unskip\space\fi MR }
\providecommand{\MRhref}[2]{%
  \href{http://www.ams.org/mathscinet-getitem?mr=#1}{#2}
}
\providecommand{\href}[2]{#2}

\end{document}